\documentclass[reqno, 10pt]{amsart}

\usepackage{amsmath,amssymb}
\usepackage[usenames, dvipsnames]{color}

\usepackage[latin1]{inputenc}

\newtheorem{theorem}{Theorem}[section]

\newtheorem{lemma}[theorem]{Lemma}

\theoremstyle{definition}
\newtheorem{remark}[theorem]{Remark}

\theoremstyle{definition}
\newtheorem{definition}[theorem]{Definition}

\theoremstyle{definition}

\newcommand\cbrk{\text{$]$\kern-.15em$]$}} 
\newcommand\opar{
\text{\,\raise.2ex\hbox{${\scriptstyle |}$}\kern-.34em$($}} 

\newcommand{\mysection}[1]{\section{#1}
\setcounter{equation}{0}}

\newcommand{\nlimsup}{\operatornamewithlimits{\overline{lim}}}

 \makeatletter
 \def\dashint{%
 \operatorname%
 {\,\,\text{\bf--}\kern-.98em\DOTSI\intop\ilimits@\!\!}}
 \makeatother

\def\bN{\mathbb{N}}

\def\bR{\mathbb{R}}

\def\fR{\mathfrak{R}}

\def\cF{\mathcal{F}}

\def\cH{\mathcal{H}}

\def\cR{\mathcal{R}}

\def\cV{\mathcal{V}}
\def\cB{\mathcal{B}}

\begin{document}

    \title[Wong-Zakai approximation]{Wong-Zakai approximation and support theorem for semilinear SPDEs with finite dimensional noise in the whole space} 
	
	\author{Timur Yastrzhembskiy}
\email{yastr002@umn.edu}
\address{127 Vincent Hall, University of Minnesota,
 Minneapolis, MN, 55455}

\keywords{SPDE, Wong-Zakai approximation, Stroock-Varadhan's support theorem, Krylov's $L_p$-theory of SPDEs}

\subjclass[2010]{35R60, 60H15}

   \begin{abstract}
    In this paper we consider the following stochastic partial differential equation (SPDE) in the whole space:
     $
	du (t, x) = [a^{i j} (t, x) D_{i j} u(t, x)  
							+ f(u, t, x)]\, dt + \sum_{k = 1}^m g^k (u(t, x)) dw^k (t).
	$
	We prove the convergence 
	 of  a Wong-Zakai type approximation scheme of the above equation
	 in the space
	 $
	   C^{\theta } ([0, T], H^{\gamma}_p (\bR^d))
	$
	in probability,
	 for some 
	$
	 \theta \in (0,1/2), 
	 			  \gamma \in (1, 2)$,
	 and  $p > 2$.
	We also prove a Stroock-Varadhan's type support theorem.
	To prove the results we combine
	 V. Mackevi\v cius's ideas from his papers on Wong-Zakai theorem 
	and the support theorem
	for diffusion processes
	with   N.V. Krylov's $L_p$-theory of SPDEs.  
	
  \end{abstract}

\maketitle

	\mysection{Introduction}
				\label{section 0}
           It was first noted  and proved by
	 E. Wong and M. Zakai 
	in  \cite{WZ_65b}
	(see also \cite{WZ_65a})
	that if $w_n$ is a sequence of  approximations of
	 a standard Wiener process $w$,
	 then, under certain conditions
	on $w_n, \sigma, x_0$,
	 the solution of the equation 
	$$
		dx_n (t) =  \sigma (x_n (t))\, dw_n (t), \quad x_n (0) = x_0
	$$
	 converges locally uniformly a.s. to the solution of the equation
	 $$
		dx (t) =  \sigma (x (t)) \, dw (t)  +  1/2\, (\sigma D\sigma) (x (t))  \, dt, \quad x(0) = x_0.
	 $$
	Here, $D \sigma$ is the derivative of $\sigma$.
	After this result was discovered, 
	there has been an extensive research around this phenomenon
	in stochastic differential equations (SDEs).
	Let us mention two articles that were the starting point of this work.
	In \cite{M_85} V. Mackevi\v cius  proved a Wong-Zakai type theorem for an SDE 
	driven by a multidimensional semimartingale
	essentially by integrating by parts in a stochastic integral,
	and later I. Gy\"ongy in \cite{G_88_SDE} 
	generalized his result by using the same technique. 
	By the way, one can modify the above approximation scheme,
	 so that the limiting equation does not have the correction term 
	$
	  1/2 \, (\sigma D\sigma) (x(t))\, dt
	$.
	In particular,  by the same method  of \cite{M_85}
	we have 
	$y_n \to y$
	 as
	 $n \to \infty$ locally uniformly  in probability,
	where
	$$
		  dy_n (t) = \sigma (y_n(t)) \, dw_n (t) - 1/2 (\sigma D\sigma) (y_n (t)) \, dt,    \quad y_n (0) = x_0, 
	$$
	     $$
			dy (t) = \sigma (y(t)) \, dw(t), \quad y(0) = x_0.
	    $$
	This fact can be used in practice to approximate certain SDEs.

	It turns out that a similar phenomenon occurs in stochastic evolution equations
	 that are, roughly speaking, described as follows:
	  \begin{align*}
		du( \omega, t) &= [ A( \omega, t) u ( \omega, t) + B (\omega, t,  u(\omega, t))] \, dt\\
			& +  \sum_{k = 1}^m G^k (\omega, t, u(\omega, t))\, dw^k (t).
	   \end{align*}
	Here,
	 $ A(\omega, t), 
				B( \omega, t, \cdot), 
								G^k (\omega, t, \cdot)$ 
	 are mappings between some Banach spaces, 
	 $A(\omega, t)$
	 is a linear map,
	 and  $w^k$  is a sequence of independent standard Wiener processes,
	where
	$k = 1, \ldots, m$,
	and  $m$ is finite.
	In this case, one expects that the sequence of the solutions of the equations
	 with regularized noise $w^k_n$
	converges to the solution of the equation with some correction term.
	There is a number of papers with Wong-Zakai type theorems
	established for various instances of operators
	 $A, B, G^k$.
	Here, we discuss only those articles that cover the case when 
	$A (\omega, t)$
	 is a second order elliptic operator. 
	The case when 
	$
	   G^k (\omega, t, \cdot)
	$ 
	is a linear operator 
	was covered in many articles such as 
	\cite{AT_84}, 
	\cite{BCF_88},
	\cite{BF_95},
	\cite{FO_14},
	  \cite{G_88_approx} - \cite{G_91}.
	The rate of convergence was obtained in 
	\cite{GS_06} 
	(see also \cite{S_05}) 
	and \cite{GS_12}.
	Results for the nonlinear case can be found in 
	\cite{DGT_12},
	\cite{N_04},
	\cite{TZ_06},
	 \cite{T_96}.
	We also  mention papers 
	\cite{BMS_95},
	\cite{CWM_01},
	\cite{HP_15},
	 \cite{HL_15}
	  on  one-dimensional parabolic SPDEs
	with infinite-dimensional noise  (i.e. $m = \infty$).
	Of these works only in \cite{HL_15} is the noise term  linear in $u$,
	and in the others it is semilinear.

	  Let $d \geq 1$
	 be an integer, and let
 	$\bR^d$ 
	be a Euclidean space of points 
	$ 
		x = (x_1, \ldots, x_d).
	$
	In this article we consider an SPDE 
	in $\bR^d$ of the following form:
	\begin{equation}
				\label{1.6}
                           du (t, x) = [a^{i j} (t, x) D_{i j} u (t, x)\, dt +   f (u, t, x)  ]\, dt  
	\end{equation}
	  $$
			           + \sum_{k = 1 }^m   g^k (u(t, x))  \, dw^k (t), 
					\quad u (0, x) =  u_0 (x).
	 $$
	Here and throughout this article we assume the summation  with respect to indexes $i, j$.
	The assumptions are stated in the Section \ref{section 1}.
	Let us just mention that $a$ is a uniformly nondegenerate bounded matrix-valued function,
	which is Lipschitz in $x$,
	and $f (u, t, x)$ is a 'first-order' term.
	We construct  Wong-Zakai approximations by regularizing  $w^k$ and
	subtracting the Stratonovich correction term:
	    \begin{equation}
				\label{1.7}
	                    du_n (t, x)  = [a^{i j} (t, x) D_{i j} u_n (t, x)  + f(u_n, t, x)
	    \end{equation}
	     $$
			        - 1/2 \, \sum_{k  =1}^m (g^k D g^k) (u_n (t, x)) ]\, dt 
														+ \sum_{k = 1 }^m   g^k (u_n (t, x))  \, dw^k_n (t),
			\, \, u_n (0, x) = u_0 (x).
	     $$
	
	 We prove the convergence of the approximation sequence
	  to the solution of \eqref{1.6}
	  in the space 
	$
	  C^{\gamma/2 - 1/p} ([0, T], X)
	$
		 in probability,
	 where
	$
		X = H^{2 - \mu}_p (\bR^d)
	$
	 is the space of Bessel potentials,  
	 and $p > 2 + d$,
	and $\gamma$ and $\mu$ are numbers
	such that 
	$
	  1 - d/p > \mu > \gamma > 2/p.
	$
	In addition, we also prove a Stroock-Varadhan's type support theorem 
	for the solution of the equation \eqref{1.6}. 
	Here, we will use V. Mackevi\v cius's approach
	 to the characterization of the topological support of a diffusion process (see also \cite{MS}).
	 In \cite{M_86} he showed that for SDEs the support theorem
	can be proved using a Wong-Zakai type approximation result 
	combined with Girsanov's theorem.
	Later, in \cite{G_88_stab}
	  (see also \cite{G_91})
	 I. Gy\"ongy, adopting methods from  \cite{M_86},
	proved a support theorem 
	 for a linear SPDE on the whole space 
	with a finite dimensional noise term. 
	In \cite{N_04}  
	and \cite{T_97} 
	 support theorems were proved
	 for SPDEs in a Hilbert space $H$ driven by an $H$-valued Wiener process.
	Let us also mention 
	 papers \cite{BMS_95} and \cite{CWM_01}
	where  support theorems were established 
	for  solutions of a nonlinear heat equation and Burgers equation
	driven by a  space-time white noise.

	We briefly explain how the aforementioned articles on Wong-Zakai problems  differ from this work.
	  The results of 
	  \cite{N_04}, 
	   \cite{TZ_06}, 
	    \cite{T_96}
	imply the convergence of the Wong-Zakai approximations
	 of the equations of  type  \eqref{1.6}
	in  the space 
	$C ([0, T], H)$,
	 where $H$ is a Hilbert space,
	and the convergence holds either in probability or in distribution.
	Perhaps, the closest to ours result was proved in \cite{DGT_12}.
	For any positive integer $p$, 
	and  $\kappa \in  (0, 1 - d/p)$,
	  pathwise convergence in the  space 
	  $
	    C^{\kappa/2} ([0, T], H^{\kappa}_p(\bR^d))
	  $ 
	 of Wong-Zakai approximations was proved for 
	 the equation \eqref{1.6}
	 with 
	$a^{i j} \equiv \delta_{i j},
		 f \equiv 0, 
		    h \equiv 0$
	 and a space-dependent nonlinearity in the stochastic term
	 (i.e.  
	$
	g^k (u (t, x))
	$
	 is replaced by 
	$
	g^k (x, u (t, x))
	$).
	In addition, the authors explained why a similar 
	result should be true for a nonlinear SPDE in the divergence form
	(i.e 
	$
	  \Delta u (t, x)
	$ 
	is replaced by
	 $
		D_{i} (a^{i j} (x) D_j u (t, x))
	$).
	However, in \cite{DGT_12} it is assumed that 
	for each $y \in \bR$
	the function
	  $x \to g^k (x, y)$
	 has a compact support.
	This rules out the case that we are interested in.
          Finally, in the papers  
	\cite{BMS_95}, 
	 \cite{CWM_01}, 
	  \cite{HP_15} 
	SPDEs are not considered on the whole space.
	
	Let us delineate the key steps of the proof of the Wong-Zakai type theorem of this article.
	First, following V. Mackevi\v cius (see \cite{M_85}),
	 we split
	 $w^k_n$
	into a 'regular part'
	 $w^k_n - w^k$
	 and the noise term
	 $w^k$.
	Since
	 $w^k_n - w^k$
	 converges to $0$ as $n \to \infty$, 
	it makes sense to integrate by parts in the integral 
	$g^k (u_n (t, x)) \, d[w^k_n (t) - w^k (t)]$. 
	It turns out that to get rid of the 'divergent' term
	from the equation \eqref{1.6},
	one needs to integrate by parts one more time.
	As a result, we obtain that a function related to the error of the approximation scheme
	satisfies a certain SPDE. 
	We show that this error converges to $0$
	 by  using an a priori estimate from N.V. Krylov's $L_p$-theory of SPDEs,
	and this allows us to prove the desired convergence.

	Finally, this author would like to thank his advisor N.V. Krylov
	for the statement of the problem, useful suggestions and attention to this work.
	The author is also grateful to the organizers of RISM school on "Developments in SPDEs in honour of G. Da Prato",
	where he had an opportunity to present the results of this paper and discuss it with other participants.

		\mysection{ Statement of the main results}
								\label{section 1}
		\textit{Basic notations and definitions.}
		Let 
		$(\Omega, \mathcal{F}, P)$
		 be a complete probability space, 
		and  let 
		$(\mathcal{F}_t, t \geq 0)$
		 be an increasing filtration of $\sigma$-fields 
		$\mathcal{F}_t \subset \mathcal{F}$
		 containing all $P$-null sets of $\Omega$.
		By $\mathcal{P}$ 
		we denote the predictable $\sigma$-field
		 generated by 
		 $(\mathcal{F}_t, t \geq 0)$.
		Let $m$ be a positive integer,
		 and 
		$\{w^i (t), t \in \bR, i = 1, \ldots, m\}$
		 be a sequence of independent standard Wiener processes
		such that 
		$w^i (t) = 0, t \leq 0 \, \forall i$.
		
		In this paper we consider only  real-valued functions.
		Denote when it makes sense
		$$
		  	D_i = \frac{\partial}{\partial x_i},    \quad
				D_{ i j} = \frac{\partial^2}{\partial x_i \partial x_j},
		\quad
			\partial_t  = \frac{\partial }{\partial t},
		   $$
		and, for a function $u$, we denote by $u_{xx}$  the matrix of second order derivatives of $u$.
		For a function
		 $f:\bR \to \bR$, 
		and any integer $k \geq 2$,
		 we denote
			$$
				D f = \frac{df}{dx}, 
							\quad D^k f = \frac{d^k f}{dx^k}. 
			$$

		Let
		 $B(\bR^d)$
		 be the space of bounded Borel functions, 
		$C^k (\bR^d)$ 
		be the space of bounded $k$ times differentiable functions 
		with bounded derivatives up to order $k$,
		$C^{\infty}_0 (\bR^d)$
		 be the space of infinitely differentiable functions with compact support,
		$C^{k + \alpha} (\bR^d)$
		 be the usual H\"older space, 
		where $k$ is a nonnegative integer, 
		and 
		$\alpha \in (0,1)$.
		For a Banach space $X$,
		 and finite $T > 0$, 	
		we denote by 
		$C^{k + \alpha} ([0, T], X)$
		the space of all $X$-valued functions 
		that are H\"older continuous in the time variable.
		In case $X = \bR$, we omit writing $X$
		inside the parenthesis. 
		For  $p \in (1, \infty]$,
		 we denote by
		 $L_p (\bR^d)$
		the space of $L_p$-integrable functions,
		and by 
		$
		  W^k_p (\bR^d)
		$ 
		and
		 $
		    W^{r, k}_p (T): = W^{r, k}_p ([0, T] \times \bR^d)
		$
		 we mean the usual Sobolev space 
		and parabolic Sobolev space (see Chapters 1 and 2 of \cite{Kr_08}).
		We introduce the spaces of Bessel potentials as follows:
		 $$
			H^{\gamma}_p (\bR^d): = (1 - \Delta)^{- \gamma/2} L_p (\bR^d), 
				\quad
															H^{\gamma}_p (\bR^d, l_2) : =  (1 - \Delta)^{- \gamma/2} L_p (\bR^d, l_2).
		$$
		Here, $\gamma \in \bR$, 
		and $l_2$ is the set of all sequences of real numbers 
		$
		   h = (h^k, k \geq 1)
		$ 
		such that
		 $
		   |h|^2_{l_2} =  \sum_{ k = 1}^{\infty} |h^k|^2 < \infty.
		$
		 For a distribution $f$
		 and a sequence of distributions
		 $
		   h = (h^k, k \geq 1),
		  $
		 we denote when it makes sense
		$$
		 	|| f ||_{\gamma, p} = || ( 1 - \Delta)^{\gamma/2} f ||_p, 
														\quad	|| h ||_{\gamma, p} = || |(1 - \Delta)^{\gamma/2} h|_{l_2} ||_p,
		$$
		where $||\cdot||_p$ stands for the $L_p$ norm.
		For a distribution $f$ and a test function
		 $
		   g\in C^{\infty}_0 (\bR^d),
		$
		we denote the action of $f$ on $g$
		by  
		   $
			(f, g).
		   $

		By $N (\ldots)$ we denote a constant depending only on
		the parameters inside the parenthesis.
		A constant $N$ might change from inequality to inequality.
		In some cases where it is clear what parameters $N$ depends on,
		we  omit listing them.

		The following facts about
		 $
		   H^{\gamma}_p (\bR^d)
		 $
		 spaces will be used in the sequel sometimes without mentioning them.
			First, for a nonnegative integer $\gamma$,
			 the spaces 
			$
			   W^{\gamma}_p (\bR^d)
			$ 
			and 
			$
				H^{\gamma}_p (\bR^d)
			$
			coincide as sets and
			have equivalent norms, 
			i.e  there exists 
			$N (d, p, \gamma) > 0$ 
			such that, 
			for all
			 $
			   u \in H^{\gamma}_p (\bR^d)
			$,
			 $$
				N  || u ||_{\gamma, p} \leq  || u ||_{ W^{\gamma}_p (\bR^d) } \leq N^{-1} || u ||_{\gamma, p}.
			 $$
			  Second, 
			$$
				|| f ||_{\gamma_1, p} \leq || f ||_{\gamma_2, p}
			$$
			if 
			$
			   \gamma_1 \leq \gamma_2, p > 1.
			$
			 The proof of these facts
			 and the detailed discussion of
			 $
			   H^{\gamma}_p (\bR^d)
			$
			 spaces can be found in Chapter 13 of \cite{Kr_08}.

		For any stopping time $\tau$, we denote 	
		$
		   \opar 0, \tau \cbrk = \{ (\omega, t): 0 < t \leq \tau(\omega) \},
		$
		 and
		  $$
			\mathbb{H}^{\gamma}_p (\tau) := L_p ( \opar 0, \tau \cbrk, \mathcal{P}, H^{\gamma}_p (\bR^d)),
				\quad
 				\mathbb{H}^{\gamma}_p (\tau, l_2) := L_p ( \opar 0, \tau \cbrk, \mathcal{P}, H^{\gamma}_p (\bR^d, l_2)),
		$$
		$$
			\mathbb{L}_p (\tau) := L_p (\opar 0, \tau \cbrk, \mathcal{P}, L_p (\bR^d)).
		$$
		
		We define stochastic Banach spaces
		 $
			\cH^{\gamma}_p (\tau).
		$
		\begin{definition}
	   	For any 
		$
		p \geq 2, \gamma \in \bR,
		$
		 and  any stopping time $\tau$,
		 we write that
		 $
		   u \in \mathcal{H}^{\gamma}_p (\tau)
		$ 
		if
		\begin{enumerate} 
		  \item $u$ is a distribution-valued process,
			 and 	$u \in  \cap_{T > 0} \mathbb{H}^{\gamma}_p ( \tau \wedge T)$,  
		 \item $u_{xx} \in \mathbb{H}^{\gamma - 2}_p (\tau)$,
		 									$u(0, \cdot) \in L_p (\Omega, \mathcal{F}_0, H^{ \gamma - 2/p }_p (\bR^d))$,
		\item there exist
		 $f \in \mathbb{H}^{\gamma - 2}_p (\tau)$
		 and 
		 $
		   h = (h^k, k \geq 1) \in 
	 	    						\mathbb{H}^{\gamma - 1}_{p} (\tau, l_2)
		$ 
		such that, for any 
		$
		   \phi \in C^{\infty}_0 (\bR^d),
		$
	 	and  any 
		$
		    t \geq 0, \omega \in \Omega,
		$
		\begin{equation}
				\label{2.3}
			\begin{aligned}
			(u (t \wedge \tau, \cdot), \phi) = &( u(0, \cdot), \phi) + 
						\int_0^{t \wedge \tau}  (f(s, \cdot), \phi) \, ds\\
		 & + 	\sum_{k = 1}^{\infty} \int_0^{t \wedge \tau} (h^k (s, \cdot), \phi) \, dw^k (s).
			\end{aligned}
		\end{equation}
		\end{enumerate}
		The norm is defined in the following way:
		$$
		    || u ||_{ \cH^{\gamma}_p (\tau) } :=  || u_{xx} ||_{\mathbb{H}^{\gamma - 2}_p (\tau)} 
		$$
		 $$
		       + || f ||_{\mathbb{H}^{\gamma - 2}_p (\tau) }   +  || h ||_{\mathbb{H}^{\gamma - 1}_p (\tau, l_2) } +
				  (E   || u (0, \cdot) ||^p_{  \gamma - 2/p, p} )^{1/p}.
		 $$

		For $u \in \cH^{\gamma}_p (\tau)$,
		 we denote  $\mathbb{D} u := f$,
		 $\mathbb{S} u := h$.
		\end{definition}
		
			\begin{remark}
			By Remark 3.2 of \cite{Kr_99}
			if
			 $
			   h \in \mathbb{H}^{\gamma}_p (\tau, l_2)
			 $, 
			for some 
			$
			 \gamma \in \bR, p \geq 2
			$, 
			then, for any 
			$
			\phi \in C^{\infty}_0 (\bR^d)
			$, 
			and any number $T > 0$,
			 the series of stochastic integrals 
			$\sum_{ k = 1}^{\infty} \int_0^t (h^k (s, \cdot), \phi) \, dw^k(s)$ 
			converges uniformly 
			in $t$ on $[0, T]$ in probability.
			\end{remark}

		\begin{remark}
				\label{remark 1.5}
		It was showed in Theorem 3.7 of \cite{Kr_99} that,
		for any 
		$
		  \gamma \in \bR, p \geq 2,
		$
		$
		\cH^{\gamma}_p (\tau)
		$
		 is a Banach space.
		Also  by the same theorem
		 if $T > 0$ is finite,
		 and $\tau \leq T$ is a stopping time,
		 then, 
		for any 
		$
		  v \in \cH^{\gamma}_p (\tau),
		$
		  \begin{equation}
					\label{1.5.0}
			|| v ||_{\mathbb{H}^{\gamma}_p (\tau) } \leq N (d,  T) || v ||_{\mathcal{H}^{\gamma}_p (\tau) }.
		  \end{equation}	  
		It follows that, for any bounded stopping time $\tau$,
		 we may replace 
		$ 
		  || u_{xx} ||_{\mathbb{H}^{\gamma - 2}_p (\tau)}
		$ by
		   $ 
			|| u ||_{\mathbb{H}^{\gamma}_p (\tau) }
		   $
		 in the definition of the norm of
		 $
		   \mathcal{H}^{\gamma}_p (\tau)
		  $
		  and obtain an equivalent norm.
		\end{remark}

		\begin{remark}
			\label{remark 1.6}
		Let 
		 $p > 2, T > 0$
		 be finite, 
		and let
		 $\theta$ and $\mu$ be numbers such that
		$
		  1 > \mu  > \theta > 2/p.
		$
		Also let
		 $\tau \leq T$
		 be a stopping time. 
		Then, by Theorem 7.2 of \cite{Kr_99}, for any
		$
			u \in \cH^{\gamma}_p (\tau)
		$,
		there exists a modification of $u$
		such that
		we have 
		$
			u \in C^{\theta/2 - 1/p} ([0, T], H^{\gamma - \mu}_p (\bR^d))
		$, 
		for any $\omega$.
		In addition,
			$$
				E || u ||^p_{ C^{\theta/2 - 1/p} ([0, T], H^{\gamma - \mu}_p (\bR^d)) } 
					\leq N (d, p, \theta, \mu, T)   E  || u ||^p_{ \cH^{\gamma}_p (\tau) }.
			$$
		Further, by the embedding theorem for 
		$H^{\gamma}_p (\bR^d)$ spaces,
		  for any non-integer $\nu$ such that
		 $
			\nu \in (0,  \gamma - \mu - d/p)
		$, 
		we have 
		$
		   u \in  C^{\theta/2 - 1/p} ([0, T], C^{\nu} (\bR^d))
		$, 
		for any $\omega$.
	\end{remark}

		\begin{definition}
		  We say that $u$ is a solution of \eqref{1.6}
		of class 
		$
		  \mathcal{H}^{\gamma}_p (\tau)
		$ 
		 if
		 $
		    u \in  \mathcal{H}^{\gamma}_p (\tau)
		$
		with
		$$
			\mathbb{D} u (t, x) = a^{i j} (t, x) D_{i j} u(t, x) + f (u, t, x),
		$$
		 $$
				 \mathbb{S} u(t, x) = (g^k (u(t, x)), k = 1, \ldots, m), 
														\quad  u (0, x) = u_0 (x)  \in H^{\gamma  - 2/p}_p (\bR^d).
		 $$
		Note that this implies that
		$
		   f(u, t, x) \in \mathbb{H}^{\gamma - 2}_p (\tau)
		$,
		and
		 $
			g^k (u (t, x)) \in \mathbb{H}^{ \gamma - 1}_p (\tau), k = 1, \ldots, m
		 $.
		\end{definition}

		\textit{Assumptions.}
		Fix some finite $p \geq 2, T > 0$.
 
		$(A1)$   
		   $
			a^{i j}(t, x) = a^{i j} (\omega, t , x),  i, j = 1, \ldots d
		    $ 
		are 
		 $\mathcal{P} \times B(\bR^d)$ - measurable functions.
		In addition, there exists 
		$\lambda > 0$ such that,
		   for all 
		$ t \geq 0, x, \xi \in \bR^d, i, j, \omega$,
		 	$$
			\lambda |\xi|^2  \leq a^{i j } (t, x) \xi_i \xi_j \leq \lambda^{-1} |\xi|^2.
		 	$$

		$(A2a) $		
		  For any $\varepsilon > 0$,
		 there exists a constant $\kappa_{\varepsilon} > 0$ such that,
		for all $ i, j, t, \omega$,
		$$
			|a^{i j} (t, x) - a^{i j} (t, y) | \leq \varepsilon
		$$
		  if $|x - y| \leq \kappa_{\varepsilon}$.

		$(A2b)$ There exists a constant $L>0$ such that, for all $i, j, t, x, y, \omega$,
		$$
			|a^{i j} (t, x) - a^{i j} (t, y)| \leq L |x - y|. 
		$$

		$(A3) (p)$  
		 $f(u, t, x)$ is a function
		defined for any 
		$
		 \omega \in \Omega, 
						u \in H^1_p (\bR^d), 
										t \geq 0, 
												x \in \bR^d
		$ 
		such that
		the following assumptions hold:

		 $(i)$ for any
		     $
		     u \in  H^1_p (\bR^d),
		    $
		  $ f (u, t, x)$  is a predictable
		      $L_p (\bR^d)$-valued function;
	 	   
		   $(ii)$ 
		    $f(0, \cdot, \cdot) \in \mathbb{L}_p (T)$;
		      
		      $(iii)$ 
		        there exists a constant $K > 0$
	     	         such that, for any 
			$u, v \in H^1_p (\bR^d), t, x,  \omega$,
			   we have
		        $$
			    ||f (u, t, x) - f(v, t, x) ||_{p} \leq K || u - v ||_{1, p}.
		         $$

		$(A4a)$ For each 
		$k \in \{ 1, \ldots, m\}$,
		 $
		  g^k (x) = c_k x, x \in \bR
		  $,
		 where  $c_k \in \bR$.

		$(A4b)$
		  For each 
		$
		   k \in \{1, \ldots, m\},
		$
		 $
		     g^k  \in C^{2} (\bR),
		  $
		 and $g^k (0) = 0$.

		$(A5) (p)$  	$u_0  \in L_p (\Omega, \mathcal{F}_0, H^{2  - 2/p}_ p(\bR^d))$.

		$(A6) (p)$    
		  For any $i$, 
		  $w^i_n$ 
		is an 
		$\mathcal{F}_t$-adapted piecewise 
		$C^1_{loc}$ approximation of $w^i$,
		and, for any $i, j$, we  denote
		$$
		     \delta w^i_n (t) = w^i (t) - w^i_n (t),
		$$
	         $$
				s^{i j }_n (t) = \int_0^t \delta w^i_n (r) \, dw^j_n (r) - \delta_{i j} t/2.
		 $$
	
		We assume that the following holds:

		$(i)$ there exists a constant $\kappa > 0$ such that,
			for any $i, \omega, t$,   
		   $
		 	 |D w^i_n (t)| \leq \kappa;
		   $		

		$(ii)$	 for any $\varepsilon \in (0, 1/2), i, j$,
			$$
			          || \delta w^i_n  ||_{ C^{ \varepsilon}   [0, T] }  +  || s^{i j}_n ||_{ C^{\varepsilon}  [0, T] }  \to 0
			$$
			 as $n \to \infty$ in probability;

		$(iii)$	for any $i, j$,
			    $$
			            \lim_{R \to \infty} \sup_{n} P (  \int_0^T  | D s^{ i j }_n (t)|^p \, dt > R) = 0.
		              $$

		\begin{remark}
				\label{remark 1.2}
		Let   
		 $h (x)$  be a Lipschitz function
		 such that $h(0) = 0$, 
		and $c(t, x), b^i (t, x)$ 
		be 
		 $\mathcal{P} \times B(\bR^d)$-measurable functions,
		 and $ f(t, x) \in \mathbb{L}_p (T), p \geq 2$.
			For any 
			$u \in H^1_p (\bR^d), t, x, \omega$, 
			we put
			$$
				f(u, t, x) =   b^i (t, x) D_i u(x) +  c (t, x) h(u(x)) +  f (t, x).
			$$
		It is easy to see that $f(u,t, x)$ satisfies $(A3) (p)$.
		\end{remark}

		\begin{remark}
		 It is proved in Appendix A that 
		the polygonal approximation defined by \eqref{4.2} 
		satisfies the assumption $(A6) (p)$.
		\end{remark}

		\textit{Remarks on the existence and uniqueness of the solution of \eqref{1.6} }
		
		We fix any finite $p \geq 2, T > 0$ and
		 assume 
		$(A1)$,	
		  $(A2a)$,
		    $(A3) (p)$,
		    $(A4a)$,
		   $(A5)(p)$.
		  Then, by Theorem 5.1 of \cite{Kr_99} with $n = 0$,
		$$
			f (z, t, x) = f (z, t, x), 
		\quad g (z, t, x) = (c_k z (x), k = 1, \ldots, m), \, \, z \in H^2_p (\bR^d),
		$$
		 there exists a unique solution $u$ 
		 of class  
		$
		    \cH^{2}_p (T)
		$
		  of the equation \eqref{1.6}.
	
		Next,  assume 
		$(A2b)$ and $(A4b)$ instead of $(A2a)$ and $(A4a)$.
		Again, by Theorem 5.1 
		with $n = -1$,
		$$
		    f (z, t, x) = f(z, t, x), \quad 
		 						g (z, t, x) = (g^k (z(x)), k = 1, \ldots, m), \, \, z \in H^1_p (\bR^d),
		 $$		  
                      the equation \eqref{1.6}
                       has a unique solution $u$  
		of class
		 $
		    \cH^{1}_p (T),
		$
		 and there exists a constant $N$ independent of $u$ such that
		\begin{equation}
				\label{2.9}
			|| u ||^p_{\mathcal{H}^{1}_p (T)} \leq 
											N (|| f (0, \cdot, \cdot)||^p_{\mathbb{L}_p (T) }  +   E || u_0 ||^p_{ 1 - 2/p }).
		\end{equation}
		We will show that 
		 $
		    u \in \mathcal{H}^{2}_p (T)
		$ 
		by using the so-called bootstrap method.
		Let 
		$$
			f (t, x) = f(u, t, x), \quad
		 					 g (t, x) = (g^k (u(t, x)), k = 1, \ldots, m).
		$$
		We claim that 
		$ 
		 f (t, x) \in \mathbb{L}_p (T),
		 $
		and 
		$
		 g (t, x) \in \mathbb{H}^{1}_p (T, l_2).
		 $
		By Theorem 3.7 of \cite{Kr_99}, 
		$
		  u (t, \cdot) \in H^1_p (\bR^d),
		 $
		for almost every
		 $
		  	 t \in [0, T], \omega.
		$
		Then, by $(A3) (p)$ and \eqref{1.5.0}
		we have
		\begin{equation}
				\label{2.10}
			E  \int_0^T || f (t, \cdot) ||^p_p  \, dt 
					\leq
				   N E \int_0^T || f (0, t, \cdot) ||^p_p  \, dt  +
					N E  \int_0^T || u (t, \cdot)||^p_{1, p} \, dt
		\end{equation}
		  $$
			\leq N	|| f (0, \cdot, \cdot)||^p_{\mathbb{L}_p (T) } +
														N || u ||^p_{\mathcal{H}^{1 }_p (T) }.
		  $$
		By the same argument 
		combined with  Lemma \ref{lemma 4.1}
		and $(A4b)$
		we get that 
		$
		   g^k (u(t, x)) \in \mathbb{H}^{1}_p (T)
		$
		 in the following way:
		\begin{equation}
				\label{2.11}
			E \int_0^T  || g^k ( u (t, \cdot) ) ||^p_{1, p} \, dt 
													\leq   N E \int_0^T || u(t, \cdot) ||^p_{1, p} \, dt \leq
																								N   || u ||^p_{\mathcal{H}^{1}_p (T) }.
		\end{equation}
		Next,  consider the equation \eqref{1.6}
		with 
		$ f (t, x)$
		 and 
		 $ g (t, x)$
		instead of 
		 $ f(u, t, x) $
		 and 
		 $
		   g (u, t, x) 
		  $
		respectively.
		Then, by  Theorem 5.1  of \cite{Kr_99} with $n = 0$
		this equation 
		has a unique solution
		 $v$ of class 
		$
		 \mathcal{H}^{ 2}_p (\tau).
		 $
		Since 
		 $
		  \mathcal {H}^{2}_p (T) \subset \mathcal {H}^{1}_p (T),
		 $
		$v$ is also a solution of \eqref{1.6} of class
		 $
		   \mathcal {H}^{1}_p (T),
		  $
		and, hence, 
		  $
		    v \equiv u
		  $	
		  as elements of 
		$
		  \cH^1_p (T).
		$
		 In addition, by \eqref{2.9} - \eqref{2.11}   we have
		 $$
			|| u ||^p_{\mathcal{H}^{2}_p (T) } \leq 
											N (|| f  ||^p_{\mathbb{L}_p (T) } +
																		 || g ||^p_{\mathbb{H}^{1}_p (T, l_2) } 
																											+ E || u_0 ||^p_{2 - 2/p, p})
		$$
		   $$
			 \leq N || f (0, \cdot, \cdot) ||^p_{ \mathbb{L}_p (T) }  +  N E || u_0 ||^p_{ 2 - 2/p, p }.
		   $$

		\textit{Statement of the main results.}
			Here is the statement of a Wong-Zakai type theorem.
 			\begin{theorem}
						\label{theorem 2.1}
			Let   	$ T > 0$, 
			$p > d + 2$, 
			$\theta \in (0, 1)$
			 be numbers.
			Assume the following:

			 $(i)$ $(A1)$, (A3)(p), (A5)(p), $(A6) (p)$ hold;
	
			$(ii)$ 
 			either $(A2a), (A4a)$ or $(A2b), (A4b)$ hold;

			$(iii)$
			 $ D^2 g \in C^{\theta} (\bR)$.
		
			Let $u$ and $u_n$ 
			be the unique solutions of class $\cH^2_p (T)$
			 of the equations \eqref{1.6} and \eqref{1.7} respectively (see Remark \ref{remark 1.7} (i)).
			Then, for any numbers $\mu$ and $\gamma$
			such that $1 - d/p > \mu > \gamma > 2/p$,
			 we have
			$$
				|| u - u_n ||_{\cV (T)} \to 0  
			$$  
			as  $n \to \infty$ in probability, 
			where 
			$
				\cV (T): = C^{ \gamma/2 - 1/p } ([0, T], H^{2  -  \mu}_p  (\bR^d)).
			$
			\end{theorem}

		Here is the statement of the support theorem.		
		\begin{theorem}
					\label{theorem 2.2}
		Assume the conditions and the notations of Theorem \ref{theorem 2.1}
		and assume that
		 $
		   a^{i j} (t, x), f (u, t, x), u_0 (x)
		$
		 are nonrandom functions.
		Take any numbers
		 $\gamma$ and $\mu$ such that  
		 $
		    1 -  d/p > \mu > \gamma > 2/p.
		 $
		Let $\cH (T)$ be the set of
		all $\bR^m$-valued functions
		 $
		   h = (h^k, k = 1, \ldots, m)
		  $
		such  that each $h^k$ is a
		  Lipschitz function on $[0, T]$,
		and $h^k (0) = 0$.
		For any
		 $h \in \cH (T)$,
		 we set $\cR (h)$
		 to be the unique solution of class 
		$
		  W^{1, 2}_p (T)
		$
		 (see Remark \ref{remark 1.7} (ii))
		 of the following equation:
		\begin{equation}
			\label{2.6}
		   \begin{aligned}
			\partial_t z (t, x) &  = a^{i j} (t, x) D_{i j}  z (t, x) 
			    +  f (z, t, x) - 1/2  \sum_{ k = 1}^m (g^k D g^k) (z (t, x))   \\
			     &  + \sum_{k = 1}^m g^k (z (t, x))  D h^k (t) , \quad z (0, x) = u_0 (x).
		   \end{aligned}
		 \end{equation}
		We denote
		 $
			\fR = \{  \cR h: h \in \cH (T)\}
		$
		 and let $\fR_{cl}$ 
		be the closure of $\fR$
		 in the space 
		$
		 \cV (T).
		$
		 Let  $u$ be the unique solution of \eqref{1.6} of class $\cH^2_p (T)$,
		$P \circ u^{-1}|_{\cV (T)}$
		 be its distribution  in  $\cV (T)$,
		and 
		 $\text{supp} \, P \circ u^{-1}|_{\cV (T)}$ 
		be the support of this measure.
		Then, 
		$
			\text{supp} \, P \circ u^{-1}|_{\cV (T)} =  \fR_{cl}.
		$
		\end{theorem}

		To prove  the main results we  use the following approximation theorem.
			\begin{theorem}
					\label{theorem 2.3}
			Assume the conditions and the notations of Theorem \ref{theorem 2.1}.
		  Assume that either  
		$\alpha = 1,  \beta = 0$
		or
		 $\alpha =  -1, \beta  = 1$,
		 and let $v$ and $v_n$ 
		be the unique solutions of class
		 $
		   \mathcal{H}^2_p (T)
		$ 
		(see Remark \ref{remark 1.7} (i))
		 of the following SPDEs:
		\begin{equation}
				\label{2.1}
			dv (t, x) =  [a^{i j} (t, x) D_{i j} v (t, x) + f (v, t, x)] \, dt 
		  \end{equation}
		 	 $$
				+  (\alpha + \beta) \sum_{ k = 1}^m g^k (v(t, x)) \, dw^k (t),
					 \quad v(0, x) = u_0 (x),
			 $$
		   \begin{equation}
					\label{2.2}
			\begin{aligned}
			dv_n (t, x) =&  [a^{i j} (t, x) D_{i j} v_n (t, x) + f (v_n, t, x) 
					          + \alpha \sum_{ k = 1}^m g^k (v_n (t, x)) \, D w^k_n (t)\\ 
						    &- (\alpha^2/2 + \alpha \beta) \sum_{k = 1}^m (g^k D g^k) (v_n (t, x))]\, dt
			                 	       + \beta \sum_{ k = 1}^m g^k (v_n (t, x)) \, dw^k (t), \\
						         &   v_n (0, x) = u_0 (x).
			\end{aligned}
			\end{equation}
			Then,  we have
			$$
				|| v - v_n ||_{\cV (T)} \to 0
			$$  
			as $n \to \infty$ in probability.
			\end{theorem}

		\begin{remark}
		  \label{remark 1.7}
		Assume the conditions of Theorem \ref{theorem 2.1}.

		$(i)$ Here, we show that the equation \eqref{2.2}
		 has a unique solution $v_n$ of class $\cH^2_p (T)$.
		The same holds for \eqref{1.7} because when
		$\alpha =  1, \beta  = 0$, we have by uniqueness
		$
		   v_n \equiv u_n
		$ 
		in
		$\cH^2_p (T)$.
		 We use the same reasoning that we used to show 
		that \eqref{1.6} has a unique solution of class
		 $\cH^2_p (T)$. 
		   This time, we set
			$$
				\bar f (z, t, x) = f (z, t, x) 
					- (\alpha^2/2 + \alpha \beta) \sum_{ k = 1}^m (g^k D g^k) (z (x))
					 + \alpha \sum_{k  = 1}^m g^k (z (x))  D w^k_n (t),  
			$$
			 $$
				\bar g (z, t, x)  = (\beta g^k ( z (x)), k = 1, \ldots, m), 
				 	\, \,  z \in H^1_p (\bR^d).
			 $$
			It is easily seen that we only need to check that $\bar f$ and $\bar g$ 
			satisfy   Assumption 5.6 of \cite{Kr_99}
			 (with $n = 0$).
			Let us show this for $\bar f$.
			For any
			  $z, v \in H^1_p (\bR^d)$,
		         and  any 
			$
			 t, \omega, 
			$
			we have
			$$
			       || \bar f (z, t, \cdot)  - \bar f (v, t, \cdot)||_p \leq
					 || f (z, t, \cdot) - f (v, t, \cdot)||_p +  
			  	\sum_{k = 1}^m || g^k (z (\cdot)) - g^k (v (\cdot)) ||_p  | D w^k_n (t)| 
			$$
			 $$
				+ 1/2 \sum_{k = 1}^m  ||(g^k D g^k) (z (\cdot)) - (g^k  D g^k) (v (\cdot)) ||_p 
				      \leq   \bar K || z - v ||_{1, p},
 			 $$ 
			 where 
			$$
				\bar K = K +   \sum_{k = 1}^m ( \kappa || D g^k ||_{\infty} 
												   +  || D g^k  ||^2_{\infty}
					  												  + || g^k  D^2 g^k  ||_{\infty}),
			$$	
			and $K$ and $\kappa$ are the constants from $(A3) (p)$,
			and $(A6) (p)$ respectively. 

			$(ii)$ Set
			      $$
				   \bar f (z, t, x) = f(z, t, x)  - 1/2  \sum_{ k = 1}^m (g^k D g^k) (z (x)) +  \sum_{ k = 1}^m g^k ( z (x)) D h^k (t),
			      $$
				$$
					\bar g (z, t, x) \equiv 0,  \, \, z \in H^2_p (\bR^d).
				$$	
			Then, by Theorem 5.1 of \cite{Kr_99}
			 (with $n = 0$) the equation \eqref{2.6}
		       has a unique solution
			 $\cR (h)$ 
			of class 
			$
			 \cH^2_p (T).
			$ 
			Since $\cR (h)$ is a nonrandom function, 
			we have
			 $
			    \cR (h) \in W^{1, 2}_p (T).
			  $
		By the same argument, in case 
		$
		  \alpha  = -1, \beta  = 1,
		$
		the exists a unique solution 
		$
		  v \in W^{1, 2}_p (T)
		$
		 of \eqref{2.1}.
		\end{remark}

		  \mysection{Auxiliary Results.}
						\label{section 3}
                        \begin{lemma}
	                        		\label{lemma 3.1}
              Let $\theta \in (0, 1)$,
	      $
		  h \in C^{1 + \theta} (\bR),
	      $
		and $h (0) = 0$.
             Let $\rho$ be a 
	  	$
		   C^{\infty}_0 (\bR)
		$  
  	 	function such that 
	  	$
		   \int \rho (y) \, dy = 1.
		$
            	Denote 
	   	$
		   \rho_{\varepsilon} (x) = 1/\varepsilon \, \rho (x/\varepsilon),
		$
	    \begin{equation}
			    \label{3.1.1}
		h_{\varepsilon} (x) = (h \ast \rho_{\varepsilon}) (x) -   (h \ast \rho_{\varepsilon}) (0),
	 \end{equation}
		where $\ast$ stands for the convolution.
		Then, the following assertions hold:

		$(i)$   for any $x \in \bR$,
	  	   $$
			| h (x) -  h_{\varepsilon} (x) | \leq 
										N (\rho)  || D h ||_{ \infty } \,  \varepsilon;
	  	   $$

		$(ii)$ for any  $k = \{0, 1, \ldots \}$,
	          $$
		  	  || D^{k+1} h_{\varepsilon}  ||_{ C^{\theta}  }
												 \leq N (\rho, \theta, k) ||  D h ||_{ C^{\theta}  } \, 1/\varepsilon^k.
	     	  $$
	   \end{lemma}
	\begin{proof}
	$(i)$ The proof is standard.

	$(ii)$  Clearly,  for any $k$,
	\begin{equation}
			\label{3.1.2}
		D^{k+1} h_{\varepsilon} (x) 
							   = 1/\varepsilon^k  \int  D^k \rho (y) \,  \, D h(x - \varepsilon y) \, dy,
	\end{equation}
	and from this the claim easily follows.
	\end{proof}

	  Denote when it makes sense
		$$
			L u (t, x) = a^{i j}(t, x)  D_{ i j} u(t, x),
		$$
		 \begin{equation}
					\label{3.0}
			M u (t, x) = a^{i j} (t, x) D_i u (t, x) D_j u (t, x).
		 \end{equation}

	       \begin{lemma}
	   			\label{lemma 2.4}
		 Assume the conditions and notations of Theorem \ref{theorem 2.3}.
		Let 
		$
		   h^{k l} :\bR \to \bR
		$
		 be a  function
		of class 
		 $
		   C^2_{loc} (\bR).
		$
		
		Denote 
		  $$
			\bar v_n (t, x) :=  v_n (t, x)  - v (t, x)  
		  $$
		   $$
			+  \alpha \sum_{ k  = 1}^m  g^k (v_n (t, x)) \delta w^k_n (t) 
														 -   \alpha^2 \sum_{k, l = 1}^m h^{k  l} (v (t, x))  s^{k  l}_n (t).
		   $$
		Then, there exist
		constants 
		$
		  N_k (\alpha, \beta), k =1, \ldots, 13
		$
		such that, for any 
		$\omega \in \Omega$, 
		$t \in [0, T]$,
		 $\psi \in C^{\infty}_0 (\bR^d)$, 
		the function $\bar v_n$
		satisfies the following equation: 
		    \begin{equation}
			\label{2.4.1}
				(\bar v_n (t, \cdot), \psi)
								   =   \int_0^t  (L \bar v_n (s, \cdot)  
																	+  \sum_{ q = 1}^{10}  N_q   F_{q, n} (s, \cdot), \psi) \, ds 
		    \end{equation}
			$$
			         +    \sum_{ r = 1}^m \int_0^t  ( N_{11} H^r_{1, n} (s, \cdot)
														 + N_{12} H^r_{2, n} (s, \cdot) + N_{13} H^r_{3, n} (s, \cdot), \psi) \, dw^r (s),
			$$
			where
			$$
				F_{1, n} (s, x) =  \sum_{ k =  1}^m (D^2 g^k) (v_n (s, x))  M v_n (s, x) \delta w^k_n (s),
			$$
			  $$
				F_{2, n} (s, x) =   \sum_{k, l = 1}^m (D^2 h^{k l}) (v (s, x))  M v (s, x) s^{k  l}_n (s),  
			  $$			  
			   $$
				F_{3, n} (s, x) =     f(v_n, s, x) -  f(v, s, x),  
			   $$
			     $$
				F_{4, n} (s, x) = \sum_{k =  1}^m f (v_n, s, x)   (D g^k) (v_n (s, x))  \delta w^k_n (s), 
			     $$
			      $$
				  F_{5, n} (s, x) =  \sum_{k, l = 1}^m f (v, s, x)  (D h^{k  l}) (v (s, x)) s^{k  l}_n (s), 
			      $$
			       $$
				    F_{6, n} (s, x)  =  \sum_{k, l = 1}^m  (g^l D g^l D g^k ) (v_n (s, x)) \delta w^k_n (s),
			       $$
			             $$
				            F_{7, n} (s, x)  =   \, \sum_{k, l = 1}^m  ( (g^l)^2 D^2 g^k) (v_n (s, x)) \, \delta w^k_n (s),
			             $$
				    $$
					  F_{8,  n} (s, x) =   \sum_{k, l = 1}^m  [(g^l D g^k) (v_n (s, x))  -  (g^l D g^k) (v (s, x))]  \, D s^{k l}_n (s),
				    $$
			                $$
				               F_{9, n} (s, x) =     \sum_{k, l = 1}^m  [(g^l D g^k) (v (s, x)) -  h^{k l} (v (s, x))]  \, D s^{k l}_n (s),
				     $$
				        $$
						F_{10, n} (s, x) =   \sum_{k, l, r = 1}^m ( (g^r)^2  D^2 h^{k l}) ( v(s, x)) \, s^{k l}_n (s),
				        $$
				         $$
					  	 H^r_{1, n} (s, x) =   g^r (v_n (s, x)) - g^r (v (s, x)),
				         $$
				          $$
					    	H^r_{2, n} (s, x) =  \sum_{k = 1}^m (g^r  D g^k ) (v_n (s, x)) \, \delta w^k_n (s),
				          $$	 
				           $$
					   	 H^r_{3, n} (s, x) =    \sum_{k, l = 1}^m ( g^r D h^{k l}) (v (s, x)) \, s^{k l}_n (s). 
				           $$	     
	
	     \end{lemma}
      	       \begin{proof}
		In the proof we assume the summation with respect to indexes
			$
			   k, l, r \in \{1, \ldots, m\}.
			$
			For any two real-valued continuous semimartingales 
			$
			  A(t), B(t), t \geq 0,
			$ 
			by
			 $<A, B>(t)$ 
			we denote their mutual quadratic variation.
				For the sake of convenience, in this proof we omit the dependence of functions on the argument $x$.

                \textit{Step 1.} 
		           Following V. Mackevi\v cius in
				 \cite{M_85} 
			         and  I. Gy\"ongy in \cite{G_88_SDE}   
				we will split the third term in the equation \eqref{2.2} into
				$
				 \alpha g^k (v_n (t)) \, dw^k (t) 
				  $
				 and 
				$
				  \alpha g^k (v_n (t) )\, d[w^k_n (t)  - w^k (t)].
				$
				 Then we will integrate by parts in the second integral.
				From this we will get a 'boundary' term, an integral and a mutual quadratic  variation term,
				which is also an integral.
				The 'boundary' term  will be subtracted from $v_n  - v$,
				and the integrals, if necessary, will be further decomposed via It\^o's formula and integration by parts.
		 
				First, we subtract the equation \eqref{2.1} from \eqref{2.2},
				and we formally write the 'stochastic part' of the difference in the following way:
				$$
					\alpha g^k (v_n (t)) \, dw^k_n (t)  + \beta g^k (v_n (t)) \, dw^k (t) 
																		- (\alpha + \beta) g^k (v (t)) \, dw^k (t)
				$$
				 $$
					= \alpha g^k (v_n (t)) \, d[w^k_n (t)  - w^k (t)] 
														+ (\alpha + \beta) [g^k (v_n (t)) - g^k (v (t))] \, dw^k (t).
				 $$
				Fix any $\psi \in C^{\infty}_0 (\bR^d)$.
				  Then, by the above,
			      \begin{equation}
					\label{2.4.3}
					(v_n (t) - v (t), \psi) = \sum_{ q = 1}^5 I^{(q)}_{ n} (t),  
			      \end{equation}
				where
				$$
					I^{(1)}_{ n } (t) = \int_0^t (L [v_n (s) - v (s)], \psi) \, ds,
				$$
				 $$
					I^{(2)}_{ n } (t) =   \int_0^t ([f (v_n, s) - f (v, s)], \psi) \, ds,
				 $$
				  $$
					I^{(3)}_{ n } (t) = \alpha \int_0^t   (g^k (v_n (s)), \psi) \, d[w^k_n (s)  - w^k (s)], 
				  $$
				   $$
					I^{(4)}_{ n } (t) =  (\alpha + \beta) \int_0^t  ([g^k (v_n (s)) - g^k (v (s))], \psi) \, dw^k (s),
				   $$
				    $$
					 I^{(5)}_{ n } (t) = - (\alpha^2/2 + \alpha \beta) \int_0^t ((g^k D g^k) (v_n (s)), \psi) \, ds.
				    $$

			Next, we assume that the support of $\psi$ is contained in some ball
			 $
			   B_R (0) : = \{ x \in \bR^d : |x| \leq R\},
			 $
			and we set
			 $$
				\phi (h) : = \int h (x) \psi (x) \, dx, \, \, h \in L_2 (B_R (0)).
			  $$
			Using Remark \ref{remark 1.5},
			it is easy to check that all the conditions of Theorem 3.1 of \cite{Kr_13} hold.
			Then, by It\^o's formula applied to 
			$
			   \phi (g^{k} (v_n (t)), t \geq 0
			$
			we obtain 
			that this process is a semimartingale, and, moreover, 
			for any 
			$\omega, t$, 
			the following holds:
			\begin{equation}
					\label{2.4.4}
				 (g^k (v_n (t)), \psi) =  (g^k (v_n (0)), \psi)  + V (t) 
		          \end{equation}
			$$
				+  \beta \int_0^t ((g^l  D g^k ) (v_n (s)), \psi) \, dw^l (s),
			$$
		  where 
		$
		  V(t), t \geq 0
		$ 
		is a process of locally bounded variation.

		Next,  by  the 
		  integration by parts formula for semimartingales we have
		\begin{equation}
				      \label{2.4.6}
			I^{(3)}_{ n } (t)
						=  - \alpha (g^k (v_n (t)), \psi) \, \delta w^k_n (t)
															 + I^{(3, 1)}_n (t) + I^{(3, 2)}_n (t),
		\end{equation}
		where
		 $$
			I^{(3, 1)}_n (t) =  \alpha \int_0^t \delta w^k_n (s) \, d(g^k (v_n (s)), \psi), 
		 $$
		   $$
			 I^{(3, 2)}_n (t) =  \alpha <  (g^k (v_n (\cdot)), \psi), w^k (\cdot) > (t).
		   $$
		By \eqref{2.4.4}  we get
		\begin{equation}
			\label{2.4.5}
					I^{(3, 2)}_n (t) = \alpha \beta \int_0^t ((g^k D g^k) (v_n (s)), \psi) \, ds.
		\end{equation}
		In the sequel we omit testing the equations with $\psi \in C^{\infty}_0 (\bR^d)$.

		Again, by It\^o's formula we get 
		$$
			I^{(3, 1)}_n (t) = \sum_{q = 1}^6 I^{(3, 1, q)}_n (t),
		$$
		where
		 $$
			I^{(3, 1, 1)}_n (t) = \alpha \int_0^t L v_n (s) D g^k (v_n (s)) \delta w^k_n (s) \, ds,
		 $$
		  $$
			I^{(3, 1, 2)}_n (t) = \alpha \int_0^t f (v_n, s) D g^k (v_n (s)) \delta w^k_n (s) \, ds,
		  $$
		   $$
			I^{(3, 1, 3)}_n (t) = \alpha^2 \int_0^t (g^l D g^k) (v_n (s)) \delta w^k_n (s) D w^l_n (s) \, ds,
		   $$
		    $$
			I^{(3, 1, 4)}_n (t)  = - \alpha (\alpha^2/2 + \alpha \beta) \int_0^t  (g^l D g^l D g^k ) (v_n (s)) \delta w^k_n (s) \, ds, 
		    $$
		     $$
			I^{(3, 1, 5)}_n (t) = \alpha \beta^2/2 \int_0^t ( (g^l)^2 D^2 g^k ) (v_n (s)) \, \delta w^k_n (s) \, ds,
		     $$
		     $$
			I^{(3, 1, 6)}_n (t) = \alpha \beta \int_0^t (g^l D g^k) (v_n (s)) \delta w^k_n (s) \, dw^l (s).
		     $$

		Next, observe that
		$$
			\delta w^k_n (s) Dw^l_n (s) =  D s^{k l}_n (s) + \delta_{k l}/2,
		$$
		and, hence,
		$$
			I^{(5)}_n (t) + I^{(3, 2)}_n (t) + I^{(3, 1, 3)}_n (t)   
		$$
		 $$
			= R_n (t) : = \alpha^2 \int_0^t (g^l D g^k) (v_n (s)) D s^{k l}_n (s) \, ds.
		 $$

			\textit{Step 2.}
		In turns out that all  integrals from Step 1 
		that we obtained after integration by parts
		 are 'under control' except $R_n (t)$.
		To handle this term we rewrite it as follows:
		  $$
			R_n (t) =  \sum_{q = 1}^3 R^{(q)}_n (t),
		    $$ 
		where
		$$
			R^{(1)}_n (t) =\alpha^2   \int_0^t  [(g^l D g^k) (v_n (s))  -  (g^l D g^k) (v (s))]  \, D s^{k l}_n (s) \, ds,
		$$
		 $$
			R^{(2)}_n (t) = \alpha^2 \int_0^t  [(g^l D g^k) (v (s)) -  h^{k l} (v (s))]  \, D s^{k l}_n (s) \, ds,
		 $$
		  $$
			R^{(3)}_n (t) = \alpha^2 \int_0^t    h^{k l} (v (s))   \, D s^{k l}_n (s) \, ds.
		  $$
		By the way, in the proof of Theorem \ref{theorem 2.3}
		the function $h^{k l}$ will be a suitable approximation of $g^l D g^k$. 

		Next, we integrate by parts in 
		$R^{(3)}_n (t)$, 
		and, since $s^{k l}_n$ has a locally bounded variation, 
		there is no mutual quadratic variation term.
		Then, we get
		\begin{equation}
			\label{2.4.7}
		    R^{(3)}_n (t) = 
							\alpha^2  h^{k l} (v (t))  s^{k l}_n (t)   +  R^{(3, 1)}_n (t),
		\end{equation}
		  $$
			R^{(3, 1)}_n (t)  : =   - \alpha^2 \int_0^t   s^{k l}_n (s) \, dh^{k l} (v (s)).
		  $$
		Applying It\^o's formula, we obtain
		$$
			R^{(3, 1)}_n (t) = \sum_{ q = 1}^4 R^{(3, 1, q)}_n (t),
		$$
		where
		  $$
			R^{(3, 1, 1)}_n (t) =  -\alpha^2 \int_0^t     L v (s)   (D h^{k l}) (v (s))  s^{k l}_n (s) \, ds,
		 $$
		  $$
			R^{(3, 1, 2)}_n (t) = - \alpha^2  \int_0^t      f (v, s)   (D h^{k l} ) (v (s))    s^{k l}_n (s)\, ds,
		  $$
		   $$
			R^{(3, 1, 3)}_n (t) =   -\alpha^2 (\alpha + \beta) \int_0^t ( g^r D h^{k l} ) (v (s)) \, s^{k l}_n (s) \, dw^r (s),
		   $$
		    $$
			R^{(3, 1, 4)}_n (t) =  -\alpha^2 (\alpha + \beta)^2/2 \int_0^t ( (g^r)^2 D^2 h^{k l}) (v (s)) \, s^{k l}_n (s) \, ds.
		    $$

		\textit{Step 3.}
		Next, we represent the sum of all the terms involving the operator $L$
		as integral of $L \bar v_n$ plus some error terms:
		$$
			I^{(1)}_n (t) + I^{(3, 1, 1)}_n (t) + R^{(3, 1, 1)}_n (t) 
		$$
		  $$
			= \int_0^t (L \bar v_n (s) - \alpha F_{1, n} (s) + \alpha^2 F_{2, n} (s)) \, ds.
		  $$

		We show how the rest of the terms on the right hand side of \eqref{2.4.1} 
		relate to the ones that appeared in the proof.
		We have
		$$
			I^{(2)}_n (t) = N_3 \int_0^t (F_{3, n} (s), \psi) \, ds,
		$$ 
		 $$
			 I^{(3, 1, 2)}_n (t) = N_4 \int_0^t (F_{4, n} (s), \psi) \, ds,
		 $$
		  $$
			 R^{(3, 1, 2)}_n (t) = N_5 \int_0^t (F_{5, n} (s), \psi) \, ds,
		  $$
		   $$
			 I^{(3, 1, 4)}_n (t) = N_6 \int_0^t (F_{6, n} (s), \psi) \, ds,
		   $$
		    $$
			   I^{(3, 1, 5)}_n (t) = N_7 \int_0^t (F_{7, n} (s), \psi) \, ds,
		    $$
		     $$
			  R^{(1)}_n (t) =  N_8 \int_0^t (F_{8, n} (s), \psi) \, ds,
		     $$ 
		      $$
			     R^{(2)}_n (t) =  N_9 \int_0^t (F_{9, n} (s), \psi) \, ds,
		      $$
		       $$
			      R^{(3, 1, 4)}_n (t) = N_{10} \int_0^t (F_{10, n} (s), \psi) \, ds, 
		       $$
			$$
				  I^{(4)}_n (t) = N_{11} \int_0^t (H^r_{1, n} (s), \psi) \, dw^r (s),
			$$
			 $$
				 I^{(3, 1, 6)}_n  (t) = N_{12} \int_0^t (H^r_{2, n} (s), \psi) \, dw^r (s),
			 $$
			  $$
				  R^{(3, 1, 3)}_n  =  N_{13} \int_0^t (H^r_{3, n} (s), \psi) \, dw^r (s).
			  $$
		Note that the terms
		 $
			I^{(3, 1, 3)}_n (t),	
						    I^{(3, 2)}_n (t),
			 							I^{(5)}_n (t)$ 
		were not lost,
		but absorbed into the term $R_n (t)$.
				       \end{proof}

		\begin{lemma}
			\label{lemma 3.4}
		Let $\alpha$ and  $\tilde \alpha$ be numbers
		 such that 
		  $
			 0 < \alpha < \tilde \alpha < 1
		$,
		 and let  $X$ be a Banach space.
		For
		 $\theta \in (0, 1)$,
		   and $t > 0$, we denote
		   $
			V^{\theta}_t  = C^{\theta} ([0, t], X)
		   $.
		    Then, for any 
			$
			  f \in V^{\tilde \alpha}_T
			 $,
			 the function 
			$
			  t \to || f ||_{ V^{\alpha}_t  }
			  $ 
			    is continuous on $[0, T]$.	
		\end{lemma}

		\begin{proof}
				Denote
		$$
			h (\xi, \nu) : = \frac{||f (\xi) - f(\nu)||_X}{|\xi - \nu|^{\alpha}}, \xi \neq \nu,
				\quad h (\xi, \xi) : = 0,
		$$
		 $$
			 x (t): = \sup_{\xi \in [0, t] } ||f (\xi)||_X, 
		 $$
		  $$
			y (t) : = \sup_{  \xi, \nu \in [0, t] } h (\xi, \nu), t \in [0, T].
		  $$
		   Take any numbers $0 < s < t < T$ and write
		\begin{equation}
			\label{3.4.2}
			||| f ||_{ V^{\alpha}_t } - || f ||_{ V^{\alpha}_s }|
				  									\leq 	| x (t) -  x (s)| + | y (t) - y (s)|.		
		 \end{equation}
		Hence, it suffices to show that both $x(\cdot)$ and $y(\cdot)$ are continuous functions on $[0, T]$.
		
		Since
		  $
		   f \in C^{\tilde \alpha} ([0, T], X)
		   $, 
		for any
		  $\xi, \nu \in [0, T]$
		such that $\xi \neq \nu$,
		we have
		\begin{equation}
			\label{3.4.1}
			| h (\xi, \nu)|
						 \leq || f ||_{ V^{\tilde \alpha} (T) } |\xi - \nu|^{\tilde \alpha - \alpha}.
		\end{equation}
		Hence,
		$h (\xi, \nu)$ is continuous on $[0, T] \times [0, T]$.
		Observe that both functions $x (\cdot)$ and $y(\cdot)$ are nondecreasing on $[0, T]$
		and, hence, at worst they have  countably many jump discontinuities.
		Then, since both $h (\cdot, \cdot)$ and $||f  (\cdot) ||_{X}$ are continuous on their domains,
		it follows that $x$ and $y$ are continuous functions on $[0, T]$.
		\end{proof}

		  \mysection{Proof of the main results}
			\textit{Proof of Theorem \ref{theorem 2.3}.}
			It is  assumed that
			 $
				i, j \in \{1, \ldots, d\}
			$, and 
			 $
			    k, l, r \in \{1, \ldots, m\}
			$,
			and  the summation with respect to these indexes is omitted.

			We fix arbitrary $\varepsilon \in (0, 1)$,
		 	 and denote by $N$ a constant independent of $\varepsilon$ and $n$.
			As before, $N$ might change from  inequality to inequality.

			Take any $R > 0$ and denote
			  $$
				\eta_n (\varepsilon): = 
				\inf\{ t \geq 0:
				 	 \sum_{k = 1}^m || \delta w^k_n  ||_{ C^{ \gamma/2 - 1/p } [0, t] } 
					+ \sum_{k, l = 1}^m ||s^{k l}_n ||_{ C^{ \gamma/2 - 1/p } [0, t] }  \geq \varepsilon \},
			  $$
			   $$
				 \gamma_n := \inf\{ t \geq 0: || v_n - v ||_{ \cV (t) } \geq 1 \}, 
			   $$
			    $$	
				\sigma_n (R) := \inf\{ t \geq 0: || v ||_{ \cV (t) } + \sum_{k, l = 1}^m \int_0^t | D s^{k l}_n (s) |^p \, ds \geq R \},  
			    $$
			       $$
				   \tau_n (\varepsilon, R) := \eta_n (\varepsilon) \wedge \gamma_n \wedge \sigma_n (R) \wedge T.
			       $$
				In Steps 1 - 5, we omit the dependence of stopping times on $\varepsilon$ and $R$.

		   		By  Remark \ref{remark 1.6},
				for any
				  $
				     \tilde \gamma \in (\gamma, \mu),
				  $
				we have
				 $
					v, v_n    \in    C^{ \tilde \gamma/2 - 1/p}  ([0, T], H^{2 -  \mu}_p (\bR^d)),
				$
				  for any $\omega$.
				 Hence, 
				by Lemma \ref{lemma 3.4}
				the process
				 $ 
				   || z ||_{ \cV (t)  }, t \in [0, T]
				$ 
				has  continuous  paths, 
				 for
				 $z \in  \{v, v_n - v\}$.
				Then, $\sigma_n$ and $\gamma_n$ are well-defined stopping times,
				and, for any $\omega$, 	
				\begin{equation}
					\label{2.5.2}
					|| v_n    -  v   ||_{ \cV (\gamma_n) } = 1, \quad
			    												 || v_n  - v ||_{  \cV (\tau_n) } \leq 1,
																							\quad || v_n ||_{ \cV (\tau_n) } \leq N.
				\end{equation}

			It follows from the definition of $\tau_n$ and
			Remark \ref{remark 1.6}
			 that, for any
			 $\omega \in \Omega$,
			 we have
			 \begin{equation}
				\label{2.5.1}
			     \sup_{s \leq \tau_n} ||z (s, \cdot) ||_{X} \leq N,
			  \end{equation}
			  where 
			  $
				 X = C^1 (\bR^d)
			  $
			 or 
			$
			   X  = H^1_p (\bR^d),
			$
			and
			 $
			    z \in \{v, v  - v_n, v_n\}.
			 $

			For any $k, l$,
			we set 
			$
			   h^{k l}_{\varepsilon}
			$
			 to be an approximation of
			$
			  g^l D g^k
			$ 
			given by \eqref{3.1.1}
			with
			 $
			   \varepsilon^{1/2}
			$
			 in place of $\varepsilon$.
			Observe that  if  $(A4a)$ holds, we have 
			$$
				h^{k l}_{\varepsilon} (x) = c_k c_l x.
			$$

			We state some properties of
			  $
			     h^{k l}_{\varepsilon}
			  $ 
			  that will be used below.
			In case when $(A4) (a)$ holds, 
			 all of  these inequalities  follow directly from the last  equality.
			Otherwise, one obtains them by using Lemma \ref{lemma 3.1}.
			By what was just said, for any $x \in \bR, k, l$, we have
			 \begin{equation}
					\label{2.5.3}
			      |(g^l D g^k - h^{k l}_{\varepsilon}) (x)| \leq           
						    					N \varepsilon^{ 1/2} |x|;
			 \end{equation}
			  \begin{equation}
					\label{2.5.4}
							| h^{k l}_{\varepsilon} (x) | \leq N |x|;
			  \end{equation}	
			 \begin{equation}
					  \label{2.5.5} 
				         || D^{q + 1 } h^{k l }_{\varepsilon} ||_{ C^{\theta} } \leq N \varepsilon^{ - q/2 },  \, \,   q  = \{0, 1, 2, \ldots \}.
			  \end{equation}

		Next, to prove the assertion of this theorem we  apply  Lemma \ref{lemma 2.4}
		with
		 $
			h^{k l} = h^{k l}_{\varepsilon}
		$ 
		and then  estimate the terms that appear in the lemma.
		Let  $\bar v_n$ be the function defined in Lemma \ref{lemma 2.4}. 
		Observe that all the conditions of  this lemma hold.
		 Then, by  the a priori estimate of Theorem 5.1 of \cite{Kr_99} with $n = 0$,
		$$
			f (t, x) = \sum_{q = 1}^{10} F_{q, n} (t, x), \quad
												   g^r (t, x) = \sum_{q = 1}^3 H^r_{q, n} (t, x),
		$$
		  we have
				\begin{equation}
				   \label{2.5.6.0}
					||  \bar v_n ||^p_{\mathcal{H}^{2}_p (\tau_n)}
														 \leq N \sum_{ q  = 1}^{13}  I_{q, n},
				\end{equation}				
		where
			$$
				I_{1, n}  =  E \int_0^{\tau_n}  || (D^2 g^k) (v_n (s, \cdot))  M v_n (s, \cdot)||_p^p  \,  |\delta w^k_n (s)|^p \, ds,
			$$
			  $$
				I_{2, n}  = E  \int_0^{\tau_n} || (D^2 h^{k l}_{\varepsilon}) (v (s, \cdot)) M v (s, \cdot)||_p^p  \, |s^{k  l}_n (s)|^p \, ds,
			  $$			  
			   $$
				 I_{3, n}  =    E \int_0^{\tau_n} || f (v_n, s, \cdot) -  f (v, s, \cdot)||_p^p \, ds,  
			   $$
			     $$
				  I_{4, n} = E \int_0^{\tau_n} || f(v_n, s, \cdot)  (D g^k) (v_n (s, \cdot)) \,  |\delta w^k_n (s)|^p \, ds, 
			     $$
			      $$
				  I_{5, n}  =   E \int_0^{\tau_n} || f (v, s, \cdot)  (D h^{k  l}_{\varepsilon}) (v (s, \cdot))||_p^p  \, |s^{k  l}_n (s)|^p \, ds, 
			      $$
			        $$
				    I_{6, n}  =  E \int_0^{\tau_n} || (g^l D g^l  D g^k) (v_n (s, \cdot))||^p_p \, |\delta w^k_n (s)|^p \,  ds, 
			        $$
			             $$
				            I_{7, n}  =   E  \int_0^{\tau_n}  || ( (g^l)^2 D^2 g^k ) (v_n (s, \cdot))||_p^p  \, |\delta w^k_n (s)|^p \, ds,
			             $$
				    $$
					  I_{8,  n}  = E \int_0^{\tau_n} || (g^l D g^k ) (v_n (s, \cdot))  -  (g^l D g^k) (v (s, \cdot))||_p^p   \, |D s^{k l}_n (s)|^p \, ds,
				    $$
			                $$
				               I_{9, n}  =   E  \int_0^{\tau_n}  || (g^l D g^k) (v (s, \cdot)) -  h^{k l}_{\varepsilon} (v (s, \cdot)) ||_p^p   \, |D s^{k l}_n (s)|^p \, ds,
				     $$
				        $$
						I_{10, n}  =  E \int_0^{\tau_n} || ((g^r)^2 D^2 h^{k l}_{\varepsilon} ) (v (s, \cdot))||^p_p \, |s^{k l}_n (s)|^p \, ds,
				        $$
				         $$
					       I_{11, n} =   E \int_0^{\tau_n}   || g^r (v_n (s, \cdot)) - g^r (v (s, \cdot)) ||^p_{1, p} \, ds,
				         $$
				          $$
					      I_{12, n}  = E \int_0^{\tau_n} || (g^r D g^k ) (v_n (s, \cdot))||^p_{1, p} \,  |\delta w^k_n (s)|^p \, ds,
				          $$	 
				           $$
					       I_{13, n}  =   E \int_0^{\tau_n}   || (g^r D h^{k l}_{\varepsilon} ) (v (s, \cdot))||^p_{1, p} \,  |s^{k l}_n (s)|^p \, ds,
				           $$
				and the operator $M$ is given by \eqref{3.0}.			

			\textit{Step 1.}	
			 We estimate the terms $I_{6, n} - I_{10, n}$.
			First, by $(A4a)$ or $(A4b)$
			  we have
			$$
				I_{6, n} + I_{7, n} \leq N E \int_0^{\tau_n} (|| v_n (s, \cdot) ||_p^p + || v_n (s, \cdot) ||_{2p}^{2p})  \, |\delta  w^k_n (s)|^p \, ds.
			$$
			Using \eqref{2.5.1}, we obtain
			\begin{equation}
					\label{2.5.7}
				I_{6, n} + I_{7, n}   \leq N \varepsilon^p.
			\end{equation}
			Similarly,
			\begin{equation}
					\label{2.5.6}
				 I_{ 8, n } \leq N E \int_0^{\tau_n} || v_n (s, \cdot) - v (s, \cdot) ||_p^p \, |D s^{k l}_n (s)|^p \,  ds,
	 		\end{equation}
		  	 \begin{equation}
					\label{2.5.8}
				I_{10, n} \leq N \varepsilon^{p/2},
			 \end{equation}
		where in the last inequality we also used \eqref{2.5.5}.
		
			Next, by \eqref{2.5.3} and \eqref{2.5.1} 
			we get
			\begin{equation}
				\label{2.5.9}
				\begin{aligned}
				I_{9, n} & \leq  
						 N  \varepsilon^{ p/2} 
					  						E \int_0^{\tau_n} || v (s, \cdot) ||_p^p \,  |D s^{k l}_n (s) |^p\, ds \\
						&\leq N \varepsilon^{ p/2}
											 E \int_0^{\sigma_n} |D s^{k l}_n (s) |^p\, ds \leq N \varepsilon^{ p/2}.
				\end{aligned}
			 \end{equation}

				\textit{Step 2.}
				We move on to the terms $I_{3, n} - I_{5, n}$.
				 By  $(A3) (p)$ we obtain
				\begin{equation}
						   \label{2.5.10}
					I_{3, n} \leq N  \,  E \int_0^{\tau_n}  || v_n (s, \cdot) - v (s, \cdot) ||^p_{1, p} \, ds. 
			           \end{equation}
			Similarly, by splitting
			 $
			    f(u, t, x)
			$
			  into
			$
			  f(u, t, x) - f (0, t, x)
			 $ 
			and 
			$
			  f(0, t, x)
			$
			 and using $(A3) (p)$,
			 we get
			\begin{equation}
					\label{2.5.11}
				I_{4, n} + I_{5, n} \leq
									 N  \varepsilon^p  \,  E \int_0^{\tau_n} (|| v (s, \cdot)||^p_{1, p} + || v_n (s, \cdot) ||^p_{1, p} + || f (0, s, \cdot) ||_p^p  ) \, ds
				  \leq  N  \varepsilon^p.
			\end{equation}	 
			Note that in the last inequality we again used \eqref{2.5.1}.

			 \textit{Step 3.} We estimate $I_{1, n}$  and $ I_{2, n}$. 
			First, by \eqref{2.5.5} we have
			        $$
				              I_{1, n} + I_{2, n} 
									\leq N \varepsilon^{  p/2}  E \int_0^{ \tau_n }  (||  M v_n (s, \cdot) ||^{p}_{ p} + || M v (s, \cdot) ||^{p}_p) \, ds.
				$$
			           Next,  by Cauchy-Schwartz inequality, for any $\omega, s$, 
				 $$
					    ||  M v_n (s, \cdot) ||^{p}_{ p}  \leq N || v_n (s, \cdot) ||^{2p}_{1, 2p}.
                                 $$         
				 Using the embedding theorem for $H^{\gamma}_p$ spaces
				 (see, for instance, Theorem 13.8.7  in \cite{Kr_08}),
				 we get, for $u \in \{ v, v_n\}$,
				$$
					|| u ||_{1, 2p} 
								\leq N || u ||_{1 + d/(2p), p}
														\leq N || u ||_{2 - \mu, p},
				$$
				where the last inequality holds, 
				 since 
				$
				 d/(2p) + \mu < 1.
				$
				By the above and \eqref{2.5.2} we get
				\begin{equation}
						     \label{2.5.12}
					  I_{1, n} + I_{2, n} \leq  N \varepsilon^{ p/2}.
				\end{equation}

		   \textit{Step 4}. We deal with $I_{11, n} - I_{13, n}$.
			The term $I_{11, n}$  will be estimated by Lemma \ref{lemma 4.2}.
			Let us check its assumptions.
			In the notations of this lemma, 
			for  
			$\omega \in \Omega$, 
			 $s \in [0, \tau_n]$,
			$ x \in \bR^d$,  we put 
			   $$
			       u (x) = v (s, x), 
			    				   \quad v (x) = v_n (s, x),
			     									  \quad 	g (x) = g^k (x). 	
			   $$
			Recall that by Remark \ref{remark 1.6}
			 we have
			 $
				v (s, \cdot), v_n (s, \cdot) \in H^1_p (\bR^d),
			$
			 for all
			 $\omega$, 
			$
			  s \in [0, \tau_n].
			$
			Also the condition \eqref{4.2.1} holds 
			due to \eqref{2.5.1}.
			Then, by Lemma \ref{lemma 4.2}
			\begin{equation}
					\label{2.5.13}
				I_{11, n} \leq N E \int_0^{\tau_n} || v (s, \cdot) - v_n (s, \cdot) ||^p_{1, p} \, ds. 
			\end{equation}

			Next, to handle the terms $I_{12, n}$ and   $I_{13, n}$ 
			we use Lemma \ref{lemma 4.1}. 
			Note that 
			$g^k (0) = 0$, 
			for each $k$, so that this lemma is applicable.
			 Then, by \eqref{2.5.5} and \eqref{2.5.1} we have 
			\begin{equation}
					\label{2.5.14}
				 I_{12, n} + I_{13, n} 
						      \leq N \varepsilon^{ p/2 } 
									\, E \int_0^{\tau_n}  (|| v_n (s, \cdot) ||^p_{1, p} + || v (s, \cdot) ||^p_{1, p}) \, ds
									\leq   N \varepsilon^{ p/2 }.
			\end{equation}

		\textit{Step 5.}
		 We combine all the  estimates 
		\eqref{2.5.6.0} - \eqref{2.5.14} and obtain that
		\begin{equation}
				\label{2.5.15}
			|| \bar v_n ||^p_{\mathcal{H}^2_p (\tau_n) } \leq
			 N \varepsilon^{ p/2}  
			+ N  E \int_0^{\tau_n} || v_n  (s, \cdot) - v (s, \cdot) ||^p_{1, p} \, (1+ |D s^{k l}_n (s)|^p)  \, ds.
		\end{equation}
		Next, we fix any stopping time $ \tau$. 
		Clearly, in \eqref{2.5.15} we may replace
		 $\tau_n$ by 
		$
		   \tau \wedge \tau_n
		$
		and, 
		this replacement will not change the constant $N$.
		Using this combined with Remark \ref{remark 1.6},
		  we get
		 \begin{equation}	
			\label{2.5.16}
			\begin{aligned}
			  & E || v_n - v ||^p_{\cV ( \tau \wedge \tau_n) } 
			 \leq N    (J_{1, n} + J_{2, n} +   \varepsilon^{   p/2 })  \\
			   &  + N E \int_0^{ \tau \wedge \tau_n} || v_n   - v  ||^p_{  \cV (s) } \, (1 + |D s^{k l}_n (s)|^p)  \, ds,
			\end{aligned}
		   \end{equation}
		    where
		  $$
			J_{1, n} =  E  ||   g^k (v_n)  \delta  w^k_n  ||^p_{ \cV (\tau_n) }, \quad
																			J_{2, n} =    E      ||   h^{k l}_{\varepsilon} ( v  ) s^{k l}_n  ||^p_{\cV (\tau_n) }.
		   $$
		Note that we were able to replace 
		$
		   || v_n  (s, \cdot) - v (s, \cdot) ||^p_{1, p}
		$
		by
		 $
		     || v_n - v||_{\cV (s)}
		 $
		 because 
		$\mu < 1$.
		Also, recall that by Lemma \ref{lemma 3.4}
		 $
			||v_n - v||_{\cV (s)}, s \geq 0
		 $
		 is a process with continuous paths,
		and, hence, the integral on the right hand side of \eqref{2.5.16} is well-defined.
	
		 Let us estimate $J_{1, n}$ and $J_{2, n}$.
		 By the product rule inequality in H\"older spaces 
		 we get 
		$$
			J_{1, n}  \leq
						 N  \varepsilon^p  E || g^{k} (v_n) ||^{p}_{ \cV (\tau_n)}.
		$$
		    Since 
		$
		2 - \mu > 1 + d/p
		$,
		  we may apply Lemma \ref{lemma 4.3},
		 and we obtain that
		$J_{1, n}$
		 is less   than
		 $$
			N  \,  \varepsilon^p (E ||v_n||^{ p }_{ \cV (\tau_n) }  
			 																  + E ||v_n||^{(2 + \theta) p}_{  \cV (\tau_n)}).
		 $$
		  By  \eqref{2.5.2}
		  we may replace the terms involving  $v_n$ by $N$. 
		  Using the same argument and \eqref{2.5.5},
		 we obtain 
		$$
			J_{2, n} \leq N \varepsilon^p  || D h^{k l}_{\varepsilon} ||^p_{ C^{ 1 + \theta} } 
																\leq N \varepsilon^{ p/2 }.
		$$
		We point out that Lemma \ref{lemma 4.3} is applicable because
		 $h^{k l}_{\varepsilon} (0) = 0 = g^k (0)$.

		 It follows from the last paragraph  and \eqref{2.5.16} that
		 \begin{equation}
					\label{2.5.17}
		   \begin{aligned}
		E x ( \tau ) \leq
							    N \varepsilon^{  p/2}  
													+ N E \int_0^{  \tau  } x (s) \, (1 + |D s^{k l}_n (s)|^p) I_{s \leq \tau_n} \, ds,	
		  \end{aligned}
		 \end{equation}
		where
		  $$
			x (s) := ||v_n - v  ||^p_{ \cV (s \wedge \tau_n) },
		  $$
		and $N$ depends on $T$.
		Then,
		 by a stochastic variant of Gronwall's inequality 
		 (see, for instance, Lemma 4 in \cite{M_85}),
		 we obtain
		\begin{equation}
					\label{2.5.18}
				E ||  v_n - v ||^p_{\cV (\tau_n) } \leq
											 N \varepsilon^{ p/2}
		\end{equation}
		because the integral
		 $
		  	 \int_0^{\tau_n} |D s^{k l}_n (s)|^p \,  ds
		 $ 
		is bounded by $R$, for any $\omega$.
 
		Next, let
		 $
		   A_n  = \{ \gamma_n <  \eta_n \wedge  \sigma_n \wedge T \}.
		 $
		Then, by \eqref{2.5.2}  we have
		$$
			P ( A_n) 
													=    E || v_n  - v ||^p_{ \cV (\tau_n) } I_{ A_n }
																					       \leq E || v_n  - v ||^p_{ \cV (\tau_n) }.
		  $$
		Therefore, by \eqref{2.5.18}
		 \begin{equation}
					\label{2.5.19}
		  P ( \gamma_n <  \eta_n \wedge \sigma_n  \wedge T) \leq   N \varepsilon^{ p/2}.
		\end{equation}
		
	  	\textit{Step 6.}
		 Finally, we prove the convergence in probability.
		By Remark \ref{remark 1.6},
		  for any $\varepsilon > 0$,
		one can choose $R >0$ such that
		$$
			P ( || v ||_{\cV (T) } \geq R) < \varepsilon.
		$$ 
		Then, by what was just said and  $(A6) (p) (iii)$,
		 for any $\varepsilon > 0$, there exists $R > 0$ such that
		\begin{equation}
			\label{2.5.20}
			\sup_n P (\sigma_n (R) < T) < \varepsilon.
		\end{equation}

		Next,
		 for any $\delta > 0$, we have
		\begin{equation}
			\label{2.5.21}
			P (|| v_n - v||_{ \cV (T)} \geq \delta) \leq 
												P (|| v_n - v||_{ \cV (\tau_n) } \geq \delta)
		\end{equation}
		  $$
				+ P ( \gamma_n < \eta_n (\varepsilon) \wedge \sigma_n (R) \wedge T)
				 + P (\eta_n (\varepsilon) < T) 
				  + P (\sigma_n (R) < T).
		  $$
		Using Chebyshov's inequality and \eqref{2.5.18},
		we get that the first term on the right hand side of \eqref{2.5.21} 
		is bounded by 
		$
		   N (\sqrt{\varepsilon}/\delta)^p.
		$	
		This combined with
		\eqref{2.5.19}, 
		$(A6) (p) (ii)$ and \eqref{2.5.20} yields
		$$
			\nlimsup_{n \to \infty} P (|| v_n - v||_{ \cV (T)} \geq \delta)
														 \leq N (\varepsilon +  (1 + \delta^{-p}) \varepsilon^{ p/2}),
		$$
		for any $\delta, \varepsilon > 0$.
		This implies the claim.

		\textit{Proof of Theorem \ref{theorem 2.1}.}
		 We set
		 $\alpha = 1, \beta  = 0$
		 and  
		let  $v$ and $v_n$ be the unique solutions of class $\cH^2_p (T)$
		of the  equations \eqref{2.1} and \eqref{2.2} respectively.
		Clearly, we have
		$v = u$ 
		and
		  $v_n = u_n$
		as elements of 
		$\mathcal{H}^2_p (T)$.
		Then, the result follows directly from Theorem \ref{theorem 2.3}.

		\textit{Proof of Theorem \ref{theorem 2.2}.}
		\textit{Step 1.} 
		First, we prove the inclusion 
		$
		  \fR_{cl} \subset P \circ u^{-1}|_{\cV (T) }.
		$
		We use the argument from \cite{M_86}, \cite{MS}.
		Let $w^k_n$ 
		be any approximation of $w^k, k \geq 1$ 
		that satisfies $(A6) (p)$.
		Fix some 
		$h \in \cH (T).$
		In the equation \eqref{2.1} and \eqref{2.2}
		 we set 
		$\alpha = -1$, 
		  			$\beta  = 1$, 
		and we replace 
		$f(z, t, x)$
		    by 
		$$
			\bar f (z, t, x): =  f (z, t, x)   + \sum_{ k = 1}^m g^k (z (x)) D h^k (t) - 1/2 \, \sum_{k = 1}^m (g^k D g^k) (z (x)), \, z \in H^1_p (\bR^d).
		$$
		Note that 
		$\bar f(z, t, x)$
		 satisfies the assumption $(A3) (p)$.
		Let 
		$
		  v \in W^{1, 2}_p (T)
		$
		 and
		 $
		   v_n \in \cH^2_p (T)
		  $
		  be the unique solutions  
		of the equations \eqref{2.1} and \eqref{2.2} respectively
		 (see Remark \ref{remark 1.7}).
		Observe that by uniqueness 
		$v = \cR (h)$ 
		in $W^{1, 2}_p (T)$,
		 and 
		$v_n$ satisfies
		  the following SPDE:
		\begin{align*}
		dz (t, x) =&  [a^{i j} (t, x) D_{i j} z (t, x) + f (z, t, x)\\
					       &  +  \sum_{ k = 1}^m g^k (z (t, x)) D h^k (t)  - \sum_{ k = 1}^m g^k (z (t, x)) \, D w^k_n (t)\\
			                 	       + & \sum_{ k = 1}^m g^k (z (t, x)) \, dw^k (t),  \quad
						            z (0, x) = u_0 (x).
		\end{align*}
		Then, by Theorem \ref{theorem 2.3}, 
		for any $\delta > 0$,
		and $n$ large, we have
		\begin{equation}
					\label{2.2.1}
			P ( || v_n - \cR (h) ||_{\cV(T) } < \delta)  > 0.  
		\end{equation}

		Next, we
		denote 
		$$
			   \bar w^k (t, n) :=  w^k (t)  - w^k_n (t) +  h^k (t),
		$$
		  $$
			\rho_n := \prod_{ k = 1}^m \exp (\int_0^T (D w^k_n (t) - D h^k (t))\, dw^k (t) 	
																		- 1/2 \, \int_0^T  |D w^k_n (t) - D h^k (t)|^2 \, dt).
		  $$ 
		Let $P_n$ be a measure on $(\Omega, \cF)$ given by
		   $$
			dP_n (\omega) :=  \rho_n (\omega) \, dP (\omega).
		  $$
	 By 
	$(A6) (p) (ii)$
	 and the fact that 
	$D h^k, k \geq 1$
	 are bounded functions on $[0, T]$,
	it follows that Novikov's condition is satisfied.
	Then, by Girsanov's theorem 
		$P_n$ is a probability measure
		 on 
		$(\Omega, \mathcal{F})$,
		 and,
		for any $n$,
		$
		 \{\bar w^k (t, n), t \geq 0, k = 1, \ldots, m\}
		$ 
		is a sequence of independent standard Wiener processes
		 with respect to
		 $
		  \{\mathcal{F}_t, t \leq T\}
		$ 
		on 
		$
		(\Omega, \mathcal{F}, P_n)
		$.
	
		Next, let 
		$
		   \mathcal{H}^2_p (T, n)
		$ 
		be the stochastic Banach space 
		defined on the probability space 
		$
		(\Omega, \mathcal{F}, P_n)
		$
		 with $w^k$ replaced by
		   $
		    \bar w^k_n, k = 1, \ldots, m
		   $.
		Observe that $v_n$ is the unique solution of class
		 $
		     \mathcal{H}^2_p (T, n)
		 $
		 of the following SPDE:
		\begin{equation}
			\label{2.2.3}
		    \begin{aligned}
			dz (t, x) &= [a^{i j} (t, x) D_{i j} z (t, x) + f(z, t, x)] \, dt\\
			   &+\sum_{k=1}^m g^k (z (t, x)) \, d\bar w^k (t, n),
											   	      \quad z (0, x) = u_0 (x).
		   \end{aligned}
		\end{equation}
		Claim that 
		\begin{equation}
			\label{2.2.4}
			P_n \circ v_n^{-1}|_{\cV (T)} 
									= P \circ u^{-1}|_{\cV (T) }.
		\end{equation}
		To prove this  we use our Wong-Zakai theorem. 
		By Theorem \ref{theorem 2.1}
		 one may replace the equations \eqref{1.6}  and \eqref{2.2.3}
		by their Wong-Zakai approximation schemes (see \eqref{1.7})
		 such that the regularized noise satisfies $(A6) (p)$.
		It follows from the method of continuity that each Wong-Zakai approximation is a fixed point
		of a contraction operator on $\cH^2_p (T)$  
		 (see the proof of Theorem 5.1 of \cite{Kr_99}).
		Now the claim follows from Picard iteration method in
		 $
		   \cH^2_p (T)
		 $ 
		and the embedding theorem for $\cH^2_p (T)$ 
		(see Remark \ref{remark 1.6}).

		Finally, by \eqref{2.2.1} 
		combined with the fact that $P_n$ is absolutely continuous with respect to $P$,
		and \eqref{2.2.4}
		we obtain that, for any $\delta > 0$,
		$$
			P \circ u^{-1}|_{\cV (T) } ( \{ z \in \cV (T): || z - \cR (h) ||_{ \cV (T) } < \delta \} ) > 0.
		$$
		This proves the announced inclusion.
	
		\textit{Step 2.}
		We prove the reverse inclusion. 
		Let $u_n$ be the unique solution of class
		 $\cH^2_p (T)$ 
		of \eqref{1.7}. 
		Then, by Theorem \ref{theorem 2.1} 
		and Portmanteau theorem
		$$
			P \circ u^{-1}|_{\cV (T) } (\fR_{cl}) 
										  \geq \nlimsup_{n \to \infty} P ( u_n \in   \fR_{cl}) = 1.
		$$
		This implies the assertion.			
		
		\mysection{Appendix A.}

		\begin{definition}
		    Denote
		    $
			h  = 1/n,
		    $
		$
			\varkappa (t) = -1 \vee t \wedge 1, t \in \bR.
		$
		We say that  the process
		   $
			w^k_n (t), t \geq 0
		   $
		    is the polygonal approximation of $w^k$
		if 
		\begin{equation}
				      \label{4.2}
		         w^k_n (t) = w^k (  (l-1) h ) + 1/h\, ( t -  l h ) \varkappa(w^k ( l h  ) - w^k ( (l-1)h )),
		\end{equation}
		  for 
		       $
			    t \in [l h, (l+1) h)
		      $,
		 and    
			$
			     l \in \{0, 1, 2, \ldots\}
			  $.  	
		\end{definition}

		The following lemma is similar to Proposition 6.3.1  of \cite{S_05}.
                          \begin{lemma}
		Let $p > 1, \varepsilon > 0$,
		  $
		   \theta \in (0, 1/2),
		  $
		$\theta' \in (0, \theta)$
		  be numbers.
		Assume that $w^i_n$ is  given by  \eqref{4.2}.
		Then, for any $i, j$,
		 the following assertions hold.
		
			$$
				(i) \,  E || \delta w^i_n   ||^p_{ C [0, T] }  \leq N (p, T, \varepsilon) n^{-  p/2  + \varepsilon }.
			$$

			  $$
				 (ii) \, E  || \delta w^i_n ||^p_{C^{1/2 - \theta}[0, T]} \leq N (p, T, \theta, \theta') n^{-  \theta' p}.
			  $$

 			$$
				(iii) \, E   || s^{i j}_n ||^p_{ C [0, T]} \leq N (p, T, \varepsilon) n^{-  p/2 + \varepsilon}.
			$$
			 
			$$
				(iv) \,   E  \int_0^T | D s^{i j }_n (t) |^p  \, dt \leq N (p, T).
			$$

			$$
				(v) \, E || s^{i j}_n ||^p_{ C^{1/2 - \theta} [0, T]} \leq N (p, T, \theta, \theta') n^{-  \theta' p}.
			$$

		\end{lemma}

		\begin{proof}
		Denote 
		$h = 1/n$, 
		$
			  t_k  = k h, k \in \{ -1, 0, 1, \ldots\}.
		$	 
		For any $a > 0$, 
		$
		   f : \bR \to \bR,
		$
		   denote
		$$
			\Delta_a f (x) = f (x + a) -  f(x),
		$$
		 $$
			 \rho_{f} ( h, T) = \sup_{t, s \in [0, T]: |t - s| \leq h} |f (t) - f(s)|.
		 $$
		For the sake of convenience,
		in  the proofs $(i), (ii)$
		we denote 
		$
		   w: = w^i, 
				    w_n  := w^i_n. 
		$

		$(i)$
		 For 
		$ 
		  t  \in [t_l, t_{l+1}),
		$
		  we have
		\begin{equation}
				\label{4.8}
		  |\delta w_n (t)|  
						\leq  |w(t) - w (t_{l-1})| + \varkappa (\Delta_{h} w (t_{l-1})),
		\end{equation}
		and 
		\begin{equation}
		   \label{4.14}			
  					\varkappa (\Delta_{h} w (t_{l-1})) \leq \Delta_{ h } (w (t_{ l - 1 })) + I_{A_{l, n}},
		\end{equation}
		where 
		$$
			A_{l, n} = \{ \Delta_{ h } w (t_{l-1}) > 1\}.
		$$
		Denote 
		$
			M = \lfloor Tn \rfloor. 
		$
		 By Chebyshov's inequality, for any $q > 0$, 
		\begin{equation}
						\label{4.1}
								P (\cup_{l = 0}^M A_{l, n}) \leq 	\sum_{ l = 0}^M P (A_{l, n}) \leq  N/h \, E |w (h)|^q \leq N h^{ q/2 -  1}.
		\end{equation}
		Then,
		\begin{equation}
		\label{4.13}
		\begin{aligned}
						& E  \max_{l = 0, \ldots, M}  |\varkappa (\Delta_{h} w (t_{l-1}))|^p  	\\
																		&\leq  N (E \rho^p_{w} (h, T) + P (\cup_{l = 0}^M A_{l, n})) \leq N h^{p/2 - \varepsilon},
		\end{aligned}
		  \end{equation}
		where in the second inequality we used the estimate of $\rho_w$ which we state below.
		   By Theorem 2.3.2 of \cite{Kr_94}, 
		 for any $\alpha > 0$,
		 there exists a positive random variable 
		$
		  N_{\alpha, T}
		$
		 such that,
		for any 
		$
		    r > 0,  \, E N_{\alpha, T}^r < \infty,
		$
		 and
		  \begin{equation}
					\label{4.3}
								\rho_{w} (\lambda, T) \leq  N_{\alpha, T} \, \lambda^{ 1/2 - \alpha}, \, \forall \omega \in \Omega,  \lambda \in [0, T].
		\end{equation}
		Then, the claim  follows from 
               \eqref{4.8}, \eqref{4.13} and \eqref{4.3}.
		By the way, similarly, for all $l$, we have
		\begin{equation}
				\label{4.19}
			E | \varkappa (\Delta_{h} w (t_l))|^p \leq N h^{p/2}.
		\end{equation}

		$(ii)$
		Fix any
		 $
		    \alpha \in (0, \theta).
		  $
		 First, we consider the case when 
		  $
			 |t - s| \geq h,
		  $
			 $
			    t, s \in [0, T].
			  $
		   We have
		   $$ 
		        1/(t-s)^{1/2 - \theta} | \delta w_n (t)  - \delta w_n (s)| 
			  \leq 2 h^{ \theta - 1/2 } ||\delta w_n  ||_{ C [0, T] },
		    $$
		and this combined with $(i)$ yields the claim.

		Next,
		we take any $t, s \in [0, T]$ such that $|t-s| < h$.
		There are  two subcases: either 
		\begin{equation}
		  \label{4.10}
					 (t, s)   \in  B_1 = \cup_{l = 0}^M \{ (t, s) \in [0, T]^2:  t, s \in [t_l, t_{ l+1 }] \}
 		\end{equation}
		or
		 \begin{equation}
		    \label{4.11}
		   \begin{aligned}
		 			(t, s) \in B_2 = &\cup_{l = 0}^{M-1} \{ (t, s) \in [0, T]^2:	  |t - s| < h, \\
																			&  t_l < s \leq t_{l+1} \leq t < t_{ l + 2} \}.
		   \end{aligned}
		   \end{equation}

		To handle  \eqref{4.10} we write 
		  \begin{equation}
				      \label{4.4}
                                     |\delta w_n (t)  - \delta w_n (s)| \leq | w_n (t) - w_n (s)| + |w (t) - w (s)|.
                       \end{equation}
		Using \eqref{4.3} and the fact that 
		$
			|t - s| \leq h,
		$
		   we get
		\begin{equation}
				\label{4.6}
			|w (t) - w (s)|  \leq N_{\alpha, T} h^{  \theta - \alpha } | t -  s|^{1/2 - \theta}.
		\end{equation}                      
	      Next,   
			by   \eqref{4.2}, \eqref{4.13} we obtain
                          \begin{equation}
					\label{4.5}
			  \begin{aligned}
                        E  & \sup_{ (t, s) \in B_1 }      | w_n (t) - w_n (s) |^p  \leq \\
															& |t - s|^p/h^p \,  E  \max_{l = 0, \ldots, M}   |\varkappa (\Delta_{h} w (t_{l - 1}))|^p   
																													     \leq N  h^{ (\theta - \alpha) p} \, |t - s|^{(1/2 - \theta)p}. 
			   \end{aligned}
		       \end{equation}
		       Then, the claim in this subcase follows from \eqref{4.4} - \eqref{4.5}.

		We move to
		\eqref{4.11}.
		Observe that
		$$
			   | w_n (t)  -  w_n (s)| 
		$$
		  $$
			\leq  |w_n (t) - w_n (t_{l+1})|  + |w_n (s) - w_n (t_{l+1})|,
		  $$
		and 
		$
		   (t, t_{l+1}), (s, t_{l+1}) \in B_1.
		$
		This combined with \eqref{4.5} and \eqref{4.6} proves the assertion for the second subcase.

		$(iii)$	
		We follow the proof of Proposition 6.3.1 of \cite{S_05}.
		First, we consider the case $i = j$.  By It\^o's formula, for any $t$, a.s. 
		$$
			|w^i (t) - w^i_n (t)|^2 = 2  \int_0^t (w^i (s) - w^i_n (s)) \, d(w^i (s)  - w^i_n (s)) + t,
		$$
		and, then,
		$$
			s^{i i}_n (t) = \int_0^t \delta w^i_n (s) \, dw^i (s) - 1/2 \, |w^i (t) - w^i_n (t)|^2.
		$$
 		Using Burkholder-Gundy-Davis inequality and assertion $(i)$,
		we get
		\begin{equation}
					\label{4.15}
			E || s^{i i}_n (t) ||^p_{ C [0, T] } \leq N E || \delta w_n^i ||^p_{ C [0, T] }  \leq N h^{p/2 - \varepsilon}.
		\end{equation}

		Now we assume $i \neq j$.
		Note that, 
		for 
		$t \in [t_k, t_{k+1}]$,
		we have
		$$
			w^i_n (t) = -\Delta_h w^i (t_{k - 1}) + w^i (t_k) + 1/h (t - t_k) \varkappa(\Delta_h w^i (t_{k-1})).
		$$
		Then, for each $\omega, t$ we may write
		\begin{equation}
					\label{4.16}
			s^{i j}_n (t) =  I_1 (t) + I_2 (t)  + I_3 (t),
		\end{equation}
		where
		$$
			I_{1} (t) = \sum_{ l = 0}^{ \lfloor tn \rfloor  }    \varkappa (\Delta_{h} w^j (t_{l - 1}))   \int_{t_l}^{t_{l+1}}  (w^i (s) -  w^i (t_l))/h  \, I_{s \leq t} \, ds,
		$$
		 $$
			I_2 (t ) =  \sum_{ l = 0}^{ \lfloor tn \rfloor  }    \varkappa (\Delta_{h} w^j (t_{l - 1}))  \Delta_h w^i (t_{l - 1}),
		$$
		$$
				I_3 (t)  = -  h^{-2}  \sum_{ l = 0}^{ \lfloor tn \rfloor  }  \int_{t_l}^{t_{l+1}} (s - t_l) I_{s \leq t }  \, ds \, \varkappa (\Delta_{h} w^j (t_{l - 1}))  \varkappa (\Delta_{h} w^i (t_{l - 1})).
		$$
		 Observe that  
		$
		  \varkappa (\Delta_{h} w^i (t_{l - 1}))
		$
		is a symmetric random variable as a composition of an odd function with a symmetric random variable.
		It follows from the Markov property of Wiener process that
		$
		   I_1 (t) 
		$ 
		is a sum of independent centered random variables, and, hence,	by  Doob's maximal inequality
		\begin{equation}
					\label{4.17}
			E \sup_{t \in [0, T]}  |I_1 (t) |^p \leq N  E | I_1 (T)  |^p \leq N h^{p/2}.
		\end{equation}
		 Let us explain how to get the second inequality in \eqref{4.17}. 
		For 
		$
		   p = \{2, 4, \ldots, \}
		$ 
		\eqref{4.17} follows from  an elementary combinatorial argument (see the proof of Lemma 6.3.2 of \cite{S_05}) combined with \eqref{4.19}.
		To prove the claim for any $p > 1$, 
			 we pick some 
			$
				k \in \bN \cup \{0\}
			$
			 such that
			$
				2k < p \leq 2k + 2,
			$
			and let $\theta$ be a number determined by the equation
			$$
				1/p = (1 - \theta)/(2k) + \theta/(2k +2). 
			$$
			Then, the inequality follows from the assertion for 
			$
			  p = 2k, 2k + 2
			$
			combined
			  with  the log-convexity of $L_p$ norms.
		Further, the same argument
		yields 
		\begin{equation}
					\label{4.18}
			E  \sup_{t \leq T} (| I_{2} (t) |^p + |I_3 (t)|^p)  \leq N  h^{p/2}.
		\end{equation}
		The claim follows from \eqref{4.15} - \eqref{4.18}.

        $(iv)$	
		By Cauchy-Schwartz inequality 
		 we get
		$$
			E \int_0^T |s^{i j}_n (t)|^p \, dt
		$$
		 $$
			 \leq
		     \sum_{k = 0}^{ \lfloor T n \rfloor } h^{-p} (\int_{t_k}^{t_{k+1}}   E |\delta w^i_n (t)|^{2p}  \, dt)^{1/2}
																						 \,  (\int_{t_k}^{t_{k+1}}  \, (E |\varkappa (\Delta_h w^j (t_{k-1}))|^{2p})^{1/2} \, dt)^{1/2}.
		$$
		This combined with \eqref{4.8} and \eqref{4.14} proves the claim.

		$(v)$	
			By Cauchy-Schwartz inequality
		we have
	     $$
	   	  E || D s^{i j }_n ||_{ L_{ \infty}[0, T] }^p \leq  h^{-p} M_{1, n} M_{2, n},
	     $$
	where 
	      $$
		M_{1, n} = (E || \delta w^i_n  ||_{  C [0, T] }^{2p})^{1/2}, 
		 \quad 
								      M_{2, n} = (E  \max_{l = 0, \ldots, \lfloor T n \rfloor}  |\varkappa (\Delta_{h} w (t_{l-1}))|^{2p} )^{1/2}.
	      $$
	By  $(i)$ and  \eqref{4.13}  
	  $$
		M_{1, n} + M_{2, n}  \leq N (p,  \varepsilon, T)  h^{ p/2 - \varepsilon}, 
	  $$
   	  Then, by the above
	\begin{equation}
		\label{4.12}
					E ||D s^{ i j }_n ||^p_{ L_{\infty} [0, T] } \leq N (p, \varepsilon, T) h^{- 2\varepsilon}.
	\end{equation}
	By the interpolation inequality 
		 (see, for example, Theorem 3.2.1 in \cite{Kr_96}),
		 for any $\lambda > 0$,
		$$
		      || s^{ i j }_n ||_{ C^{1/2 - \theta} [0, T]}
			 \leq  	N (\theta, T) (h^{1/2 + \theta} ||D s^{i j}_n||_{ L_{\infty} [0, T]} 
				+ h^{ \theta - 1/2 } || s^{i j}_n ||_{ C[0, T]}). 
		$$
		 We  finish the proof by combining this with  $(iii)$  and \eqref{4.12}.
	  	\end{proof}

		\mysection{Appendix B.}

      	  \begin{lemma}
		\label{lemma 4.1}
 	Let
        $
	  g \in C^1_{loc} (\bR),
	$
	   $
		D g \in L_{\infty} (\bR),
	   $
  		$g(0) = 0$,
	 and 
	 $u \in H^1_p (\bR^d)$.
	Then, 
	$$
		|| g (u (\cdot) ) ||_{ 1, p } \leq N ( d, p) || D g ||_{\infty}  || u ||_{ 1, p }.	   
	  $$		
	\end{lemma}
	 
	\begin{proof}
	Recall that the spaces 
	 $W^1_p (\bR^d)$ 
	and
	 $H^1_p (\bR^d)$ 
	coincide as sets
	 and have equivalent norms.
	By this and the chain rule in $W^1_p (\bR^d)$
	we may write
	$$
		|| g (u (\cdot)) ||_{ 1, p }  
							\leq N (d, p) ( || g ( u (\cdot) ) ||_p 
												  +  || D_i u (\cdot) D g (u (\cdot))||_p)
	$$
	  $$
		\leq N || D g ||_{\infty} (|| u ||_p + || D_i u ||_{p})
	  $$
	   $$
								 = N ||D g ||_{\infty} ||u ||_{W^1_p (\bR^d)}
														 \leq N || D g ||_{\infty} || u ||_{1, p}. 
	   $$
	\end{proof}

	\begin{lemma}
		\label{lemma 4.2}
	 Let 
	$
		 u, v \in H^1_p (\bR^d)
	$
	and
	assume that 
		\begin{equation}
		   \label{4.2.1}
		  || D g||_{\infty} + || D^2 g ||_{\infty} \leq K,
											\quad    || D_i u ||_{\infty}   \leq R.
		\end{equation}
		 Then, 
		$$
			|| g(u (\cdot) ) - g(v (\cdot) ) ||_{1, p} \leq N (d, p) K  ( 1 +  R) || u  - v ||_{1, p}.
		$$
	\end{lemma}
	           \begin{proof}
		By the argument of the proof of Lemma \ref{lemma 4.1}  we have 
				$$
					 || g(u (\cdot) ) - g(v (\cdot)  ) ||_{ 1, p }  
				$$
				 $$
					\leq  N (d, p) (|| g (u (\cdot) ) - g (v (\cdot) ) ||_p + 
															 ||  D_i u (\cdot)  D g (u (\cdot) ) -  D_i v (\cdot) D g (v (\cdot) )  ||_p)
				 $$
				  $$
					\leq 	N K ||u - v||_p 
									+ N || D_i u (\cdot) [ D g (u (\cdot) ) -  D g (v (\cdot) ) ] ||_p
				  $$
			           $$
						 +  N || D g(v (\cdot) ) [ D_i u (\cdot) - D_i v (\cdot) ]||_p.
				   $$
				By this and \eqref{4.2.1}
				we obtain the assertion of the lemma.

		\end{proof}

		\begin{lemma}
					\label{lemma 4.3}
			Let  
			$ p > d$, 
			 $\delta \in (0,1 )$, 
				 $\gamma \in (d/p, 1), \tau > 0$
			 be  numbers.
			Let $g: \bR \to \bR$ be a  function 
			such that
			  $D g \in C^{1 + \theta} (\bR)$,
			and $g(0) = 0$.
			Denote  
		             $$  
				   \cB:  =  C^{\delta} ([0, \tau], H^{1 + \gamma}_p (\bR^d) )
			     $$
			and take any $u \in \cB$.
				Then, there exists a constant $N$ independent of $g$ and $u$ such that
			$$
			   	||g(u)||_{\cB } \leq N  
								 || D g ||_{ C^{1 + \theta} } 
											 ( ||u||_{ \cB  } + ||u||^{2 + \theta}_{ \cB } ).
			$$
			\end{lemma}

			\begin{proof}	
			For the sake of  convenience,
		 we denote 
               $ u_t = u (t, \cdot)$,
		and we omit the dependence of $u$ on the spatial variable $x$.
			 We set
			$
		  	     N_g =  || D g ||_{ C^{1 + \theta} }.
			$

			First, we prove the supremum norm estimate.
			Note that
			$
			  H^{1 + \gamma}_p (\bR^d)
			$ is embedded in 
			$
			  C (\bR^d)
			$
			since $ 1 + \gamma > d/p$
			(see, for instance, Theorem 13.8.1 of \cite{Kr_08}).
			Then, by this and Corollary 3 combined with  Remark 3 of Section 5.3.7 of \cite{RS_96} we have
			\begin{equation}
				E \sup_{t \leq T} || g (u) ||_{ 1 + \gamma, p } \leq N_g (|| u ||_{ \cB } + || u ||^{1 + \gamma}_{ \cB }).
			\end{equation} 	
					
			Next, take any 
			$
				s, t \in [0, \tau]
			$
			such that 
			  $s \neq t$.
			By the fact that  
			 $(1 - D_i)$ 
			is a strongly elliptic differential operator of order $1$
			we obtain (see Theorem 13.3.10 of \cite{Kr_08})
			\begin{equation}
			   \label{4.3.1}
			    || g ( u_t  ) - g ( u_s) ||_{1+ \gamma, p } \leq  N (J^{(1)} + J^{(2)}),
			\end{equation}
			where
			  $$
			          J^{(1)} =    || g ( u_t ) - g( u_s )  ||_{ \gamma, p },
			 $$
			  $$
				J^{(2)} =  || D g ( u_t  ) D_i u_t  - D g( u_s ) D_i u_s  ||_{ \gamma, p }.
			  $$

			  Recall that, by the elementary embedding (see Section \ref{section 1}) we may replace
			 $\gamma$ by $1$ in the expression for  $J^{(1)}$. 
			      Since 
				$\gamma > d/p$, 
				by the embedding theorem for 
				  $H^{\mu}_p (\bR^d)$ spaces 
				 we have
				$$
				      \sup_{t \leq T} ||D_i u_t ||_{\infty} \leq N || u ||_{ \cB }.
				$$
			 	Then, by Lemma \ref{lemma 4.2} and what was just said we obtain
			        \begin{equation}
				   \label{4.3.2}
					J^{(1)} \leq  N N_g (1  + || u ||_{ \cB })   || u_t - u_s||_{1, p}.
			          \end{equation}

			Next, by the triangle inequality 
			\begin{equation}
					\label{4.3.5}
				J^{(2)} \leq   J^{ (2, 1) } + J^{ (2, 2)}  + J^{ (2, 3) },
			\end{equation}
			   where
			 $$
			       J^{(2, 1)} =   || (D g ( u_t ) - D g (0)) ( D_i u_t  - D_i u_s)  ||_{ \gamma, p }, 
			 $$
			  $$
				J^{(2, 2)} = 	  ||  D_i u_s (D g( u_t ) - D g( u_s )) ||_{\gamma, p},
			  $$
			   $$
				J^{(2, 3)} = |D g (0)| || u_t - u_s ||_{1 + \gamma, p}.
			   $$
			It is well-known that
			$
			   H^{\gamma}_p (\bR^d)
			$ 
			is a multiplication algebra
                          because $\gamma > d/p$ 
                          (see, for example,  Theorem 1 in Section 4.6.1 of \cite{RS_96}).
			Then, we get
			 $$
				J^{(2, 1)} \leq  N || D g (u_t) - D g(0) ||_{\gamma, p} || u_t - u_s ||_{1 + \gamma, p},
			$$
	  		 $$
				J^{(2, 2)} \leq N || u_s ||_{1 + \gamma, p} || D g (u_t) - D g(u_s) ||_{\gamma, p}.
			 $$  
			  To handle $J^{(2, 1)}$
				 we estimate 
			$|| D g (u_t)   - D g(0) ||_{1, p}$  
			 via Lemma \ref{lemma 4.1}.
			By this we have   	
				  \begin{equation}
					\label{4.3.6}
					  J^{(2, 1)} \leq N N_g || u_t ||_{ 1, p} || u_t - u_s ||_{1 + \gamma, p}.
				  \end{equation}			
			Next, using the fact that $D g^2 \in C^{\theta }$,  we obtain
			$$
				|| D  g (u_t) - D g (u_s)||_{\gamma, p} \leq || D g (u_t) -  D g(u_s)||_{1, p} 
			$$
			  $$
				\leq N || D g (u_t) -  D g (u_s)||_{ p} 
			  $$
			     $$
				 + N   || D^2 g (u_s) (D_i u_t - D_i u_s) ||_p 
			     $$
			      $$
				+ N  ||D_i u_t (D^2 g (u_t) - D^2 g (u_s) ) ||_p 
			      $$
			       $$
				   \leq N N_g ||u_t - u_s||_{1 + \gamma, p} +   N N_g ||  u_t ||_{ 1, p } || u_t - u_s||^{ \theta }_{ \infty}.
			       $$
			Again,  by the embedding theorem 
			 we may replace
			  $|| u_t - u_s||_{ \infty}$
			 by
			    $|| u_t - u_s ||_{ 1 + \gamma, p}$.
			Then, by the above we have
			\begin{equation}
					\label{4.3.7}
			  J^{(2,2)}
						\leq  N N_g  ||u_s||_{ 1 + \gamma, p} 
													 ( ||u_t - u_s||_{1 + \gamma, p} 
																				+   ||u_t||_{1 + \gamma, p} || u_t - u_s||^{\theta}_{ 1 + \gamma, p}) 
			 \end{equation}
			The assertion follows from \eqref{4.3.1} - \eqref{4.3.7}.  			
			\end{proof}

                                      \end{document}